\documentclass[11pt]{article}
\usepackage{amsmath,latexsym,amssymb}
\renewcommand{\Bbb}[1]{\mathbb{#1}}
\usepackage{tikz}
\usetikzlibrary{arrows,calc}
\usepackage{pdflscape}
\usepackage{rotating}
\usepackage{graphics}
\usepackage{enumerate}
\newif\ifpdf
\pdffalse
\pdftrue
\ifpdf
\def\fext{pdf}
\else
\def\fext{eps}
\fi
\providecommand{\1}{\mathbf{1}}
\newcommand{\ignore}[1]{}
\newcommand{\mfalls}{\quad\mbox{if \;}}
\newcommand{\msonst}{\quad\mbox{else}}
\renewcommand{\cases}[1]{\left\{\begin{array}{rl}#1\end{array}\right.}

\newcommand{\mbs}[1]{\mbox{ \;#1\; }}

\newcommand{\mf}{\quad\mbox{\;for \;}}
\newcommand{\mfa}{\quad\mbox{\;for all \;}}

\newcommand{\mfs}{\quad\mbox{a.s.}}

\renewcommand{\P}{\mathbf{P}}
\newcommand{\E}{\mathbf{E}}

\def\given{\hspace{0.8pt}|\hspace{0.8pt}}
\def\Given{\hspace{1.0pt}\big|\hspace{1.0pt}}
\def\ggiven{\hspace{1.2pt}\Big|\hspace{1.2pt}}

\newcommand{\limn}{\stackrel{n \rightarrow \infty}{\longrightarrow}}

\newcommand{\N}{{\mathbb{N}}}
\newcommand{\Ray}{{\mathsf{Ray}}}
\newcommand{\CT}{{\mathcal{T}}}

\newcommand{\CN}{{\mathcal{N}}}
\newcommand{\EE}{\mathbb{E}}
\newcommand{\PP}{\mathbb{P}}

\newcommand{\Z}{{\mathbb{Z}}}

\newcommand{\bec}{\begin{equation}}
\newcommand{\eec}{\end{equation}}
\newcommand{\bac}{\begin{eqnarray}}
\newcommand{\eac}{\end{eqnarray}}
\newcommand{\be}{\begin{displaymath}}
\newcommand{\ee}{\end{displaymath}}
\newcommand{\ba}{\begin{eqnarray*}}
\newcommand{\ea}{\end{eqnarray*}}

\newcommand{\equ}[1]{(\ref{#1})}

\newcommand{\DS}{\displaystyle}

\newtheorem{proposition}{Proposition}[section]
\newtheorem{theorem}[proposition]{Theorem}

\newtheorem{lemma}[proposition]{Lemma}

\newtheorem{uremark}[proposition]{Remark}

\newtheorem{uexample}[proposition]{Example}

\newcommand{\Section}[1]{\section{#1}\setcounter{figure}{0}\setcounter{table}{0}\setcounter{equation}{0}}

\def\qed{\mbox{$\Box$}}

\newenvironment{proof}{\par\noindent{\bf Proof.\ }}{\hfill\qed\\ }
\newcommand{\weight}{\mathop{\mbox{\sf weight}}}

\newtheorem{udefi}{Definition}[section]
\newtheorem{utheo}{Theorem}
\newtheorem{usatz}[udefi]{Satz}
\newtheorem{uprop}[udefi]{Proposition}

\newtheorem{ubemerkung}[udefi]{Bemerkung}
\newtheorem{ukorollar}[udefi]{Korollar}

\def\CL{\mathcal{L}}

\newcommand{\restbig}[1]{\raisebox{-0.45em}{$\Big|_{\scriptstyle
#1}$}}

\usepackage{color}
\definecolor{Red}{rgb}{1,0,0}

\begin{document}
\title{Biased Random Walk on Spanning Trees of the Ladder Graph}
\author{
Nina Gantert\\
Fakult\"{a}t f\"{u}r Mathematik\\
Technische Universit\"{a}t M\"{u}nchen\\
Boltzmannstr. 3\\
85748 Garching\\
Germany\\
gantert@ma.tum.de
\and Achim Klenke\\
Institut f\"{u}r Mathematik\\
Johannes Gutenberg-Universit\"{a}t Mainz\\
Staudingerweg 9\\
55099 Mainz\\Germany\\
math@aklenke.de}

\date{\small Version from 22.02.2023}

\maketitle
\begin{abstract}
We consider a specific random graph which serves as a disordered medium for a particle performing biased random walk. Take a two-sided infinite horizontal ladder and pick a random spanning tree with a certain edge weight $c$ for the (vertical) rungs. Now take a random walk on that spanning tree with a bias $\beta>1$ to the right.
In contrast to other random graphs considered in the literature (random percolation clusters, Galton-Watson trees) this one allows for an explicit analysis based on a decomposition of the graph into independent pieces.

We give an explicit formula for the speed of the biased random walk as a function of both the bias $\beta$ and the edge weight $c$. We conclude that the speed is a continuous, unimodal function of $\beta$ that is positive if and only if $\beta <  \beta_c^{(1)}$ for an explicit critical value $\beta_c^{(1)}$ depending on $c$. In particular, the phase transition at $\beta_c^{(1)}$ is of second order.

We show that another second order phase transition takes place at another critical value $\beta_c^{(2)}<\beta_c^{(1)}$ that is also explicitly known: For $\beta<\beta_c^{(2)}$ the times the walker spends in traps have second moments and (after subtracting the linear speed) the position fulfills a central limit theorem. We see that
$\beta_c^{(2)}$ is smaller than the value of $\beta$ which achieves the maximal value of the speed. Finally, concerning linear response, we confirm the Einstein relation for the unbiased model ($\beta=1$) by proving a central limit theorem and computing the variance.
\end{abstract}
\Section{Introduction and Main Results}
\label{S1}
\subsection{Introduction}
\label{S1.1}
This paper studies a very specific model for transport in a disordered medium.
Biased random walks in random environments and on random graphs have been investigated intensively over the last years. The most prominent examples are biased random walk on supercritical percolation clusters, introduced in \cite{BarmaDhar}
 and biased random walk on supercritical Galton-Watson tree, introduced in \cite{LPP}. We refer to \cite{GBAFri} for a survey. Another specific model which has found a lot of recent interest in the physics literature is the random comb graph, see \cite{KotakBarma}, \cite{WhiteBarma}, \cite{BalakBroeck}, \cite{DemaerelMaes}.
In the presence of traps in the medium, there are often three regimes of transport, see for instance \cite{KotakBarma} and the references therein.
\begin{itemize}
\item[1.] The Normal Transport regime for small values of the bias: the walk has a positive linear speed and, when subtracting the linear speed, it is diffusive.
\item[2.] The Anomalous Fluctuation regime for intermediate values of the bias: the walk still has a positive linear speed but the diffusivity is lost.
\item[3.] The Vanishing Velocity regime (aka subballistic regime):
the speed of the random walk is zero if the bias is larger than some critical value, due to the time the random walk spends in traps.
\end{itemize}
The Normal Transport regime together with the Anomalous Fluctuation regime are also known as the ballistic regime.
For biased random walk on supercritical Galton-Watson trees, these statements have been proved in \cite{LPP}. For biased random walks on supercritical percolation clusters, the existence of the critical value separating the ballistic regime from the  Vanishing Velocity regime
was shown in \cite{FriHam}, whereas the earlier works \cite{Sznitperc} and \cite{BGP} gave the existence of a zero speed and a positive speed regime. In the ballistic regime, one may ask about the behaviour of the linear speed as a function of the bias. Is the speed increasing as a function of the bias?
This question is also interesting in disordered media without ``hard traps'', for instance Galton-Watson trees without leaves or the random conductance model (with conductances that are bounded above and bounded away from $0$). In that case, there is no Vanishing Velocity regime.
Monotonicity of the speed for biased random walks on supercritical Galton-Watson trees without leaves is a famous open question, see
\cite{LyonsPeresPemantle1997}.
We refer to \cite{Aidekon2014} and \cite{BenArousFriberghSidorvacius2014} for recent results on Galton-Watson trees and \cite{BergerGantertNagel2019} for a counterexample to monotonicity in the random conductance model. The Normal Transport regime for biased random walks on supercritical percolation clusters has been established in \cite{FriHam}, \cite{Sznitperc},  \cite{BGP}.
Limit laws for the position of the walker have been investigated both in the Anomalous Fluctuation regime and in the Vanishing Velocity regime in several examples, see \cite{Hammond}, \cite{BFGH}, \cite{GMM2}, \cite{MeiLue}, \cite{BowditchCroydon2022}.
For biased random walk on a supercritical percolation cluster, the conjectured picture of the speed as a function of the bias is as in Figure \ref{F1.3}.
However, there is no rigorous proof that the speed is a unimodal function of the bias.
 Here, we consider a random graph given as a (uniformly chosen) spanning tree of the ladder graph, parametrized by the density of vertical edges. In this case, we can give an explicit formula for the speed of the biased random walk, see \eqref{E1.12}.
 In particular, we have an explicit critical value $\beta_c^{(1)}$ for the bias such that
 the speed is positive for $\beta\in\big(1,\beta_c^{(1)}\big)$ and is zero for $\beta\geq\beta_c^{(1)}$.
From this formula, we see that the speed is a unimodal function of the bias, see Figure \ref{F1.3}.
The formula also allows to study the dependence of the speed on the density of vertical edges.
We show that for an explicit value $\beta_c^{(2)}< \beta_c^{(1)}$, a central limit theorem holds for $\beta < \beta_c^{(2)}$. This establishes the Normal Transport regime for our model and is not surprising as the same statement is true for other biased walks on random graphs, see \cite{Sznitperc}, \cite{BGP}, \cite{FriHam}, \cite{GMM2}. In contrast to these examples, the critical value  $\beta_c^{(2)}$ is explicit in our case.
For the unbiased case, we even have a quenched invariance principle. By computing the variance, we confirm the Einstein relation for our model.
It has been said (but we do not have a written reference for this conjecture)  in the general setup that the critical
$\beta_c^{(2)}$ for the existence of second moments and for the validity of a central limit theorem is the value of $\beta$ where the speed is maximal. However, this is not true in our example. We show that $\beta_c^{(2)}$ is strictly smaller than the value where the speed is maximized.\\
Our proofs rely on a decomposition of the uniform spanning tree due to \cite{Klenke2017}, on explicit calculations for hitting times using conductances, on regeneration times and some ergodic theory arguments.
The decomposition of the spanning tree allows for an interpretation as a trapping model in the spirit of \cite{GBAFri, Bowditch2019, BetzMeinersTomic2023}.

\subsection{Definition of the model}
\label{S1.2}
To define our model of biased random walk on a random spanning tree, we need to introduce two things: (1) the random spanning tree and (2) the random walk on it. We begin with the random spanning tree.

\textbf{Random spanning tree}\par
Consider the two-sided infinite ladder graph $L=\big(V^L,E^L\big)$ with vertex set $V^L=\{0,1\}\times\Z$ and edge set
$$E^L=\big\{z_m,\,h_{0,m},\,h_{1,m}:\;m\in\Z\big\}.$$
Here the
$$h_{i,k}=\{(i,k),(i,k+1)\},\qquad i\in\{0,1\},\;k\in\Z,$$
are the horizontal edges and the
$$z_k:=\{(0,k),(1,k)\},\qquad k\in\Z,$$
are the vertical edges.
See Figure~\ref{F1.1}.
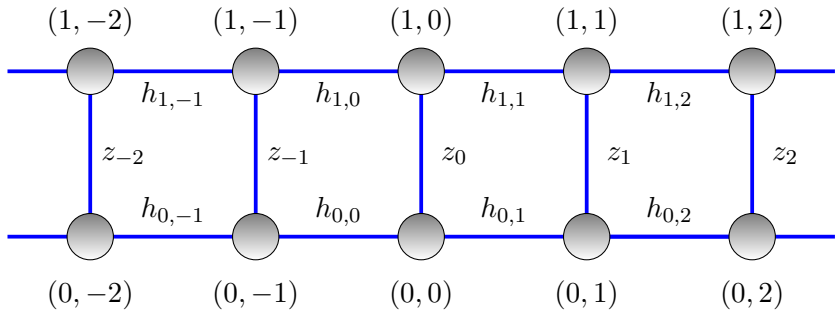
\begin{figure}[ht]
\centerline{
\begin{tikzpicture}[scale=2.2]
\def\myrad{0.14cm}
\draw[color=blue, line width = 0.5mm] (-0.5,0) -- (4.5,0);
\draw[color=blue, line width = 0.5mm] (-0.5,1) -- (4.5,1);
\draw[color=blue, line width = 0.5mm] (3,0) -- (4,0);
\foreach \x in {0,...,4}
  \draw[color=blue, line width = 0.5mm] ({\x},0) -- ({\x},1);
\foreach \x in {0,...,4}
  \foreach \y in {0,1}
  \shadedraw ({\x},{\y}) circle(\myrad);
\draw (0.5,+0.3) node[anchor=north]{$h_{0,-1}$};
\draw (0.5, 1.0) node[anchor=north]{$h_{1,-1}$};
\draw (1.5,+0.3) node[anchor=north]{$h_{0,0}$};
\draw (1.5, 1.0) node[anchor=north]{$h_{1,0}$};
\draw (2.5,+0.3) node[anchor=north]{$h_{0,1}$};
\draw (2.5,+1.0) node[anchor=north]{$h_{1,1}$};
\draw (3.5,+0.3) node[anchor=north]{$h_{0,2}$};
\draw (3.5, 1.0) node[anchor=north]{$h_{1,2}$};
\draw (0.2, 0.6) node[anchor=north]{$z_{-2}$};
\draw (1.2, 0.6) node[anchor=north]{$z_{-1}$};
\draw (2.2, 0.6) node[anchor=north]{$z_{0}$};
\draw (3.2, 0.6) node[anchor=north]{$z_{1}$};
\draw (4.2, 0.6) node[anchor=north]{$z_{2}$};
\draw (0.0, -0.2) node[anchor=north]{$(0,-2)$};
\draw (1.0, -0.2) node[anchor=north]{$(0,-1)$};
\draw (2.0, -0.2) node[anchor=north]{$(0,0)$};
\draw (3.0, -0.2) node[anchor=north]{$(0,1)$};
\draw (4.0, -0.2) node[anchor=north]{$(0,2)$};
\draw (0.0, 1.15) node[anchor=south]{$(1,-2)$};
\draw (1.0, 1.15) node[anchor=south]{$(1,-1)$};
\draw (2.0, 1.15) node[anchor=south]{$(1,0)$};
\draw (3.0, 1.15) node[anchor=south]{$(1,1)$};
\draw (4.0, 1.15) node[anchor=south]{$(1,2)$};
\end{tikzpicture}
}
\caption[]{A finite section of the simple ladder graph}
\label{F1.1}
\end{figure}

For $n\in \N$, let
$$V^L_n:=\{0,1\}\times\{-n,\ldots,n\}\subset V^L$$
and let $E^L_n$ denote the induced set of edges. Finally, let $L_n:=(V^L_n,E^L_n)$ denote the induced finite subgraph of $L$.

Let $\mathsf{ST}(L)$ denote the set of all spanning trees of $L$. That is, each $t\in\mathsf{ST}(L)$ is a subset of $E^L$ such that the graph $(V^L,t)$ is connected but has no cycles. Analogously, define $\mathsf{ST}(L_n)$.

Let $c>0$ be a parameter of the model. We attach a weight $\weight(z_m)=c$ to each vertical edge $z_m$ and $\weight(h_{i,m})=1$ to each horizontal edge $h_{i,m}$.

Denote by $\PP^c_n$ the weighted spanning tree distribution on $\mathsf{ST}\big(L_n\big)$, that is
\begin{equation}
\label{E1.01}
\PP^c_n[\{t\}]=\frac{\weight(t)}{\weight\big(\mathsf{ST}\big(L_n\big)\big)}\mf t\in\mathsf{ST}\big(L_n\big).
\end{equation}
By taking the limit $n\to\infty$, we get (in the sense of convergence of finite dimensional distributions)
$$\PP:=\PP^c:=\lim_{n\to\infty}\PP^c_n.$$
By a standard recurrence argument, $\PP$ is concentrated on connected graphs. That is, $\PP[\mathsf{ST}(L)]=1$.

Let $\EE$ denote the expectation with respect to $\PP$ and let $\CT$ be the generic random spanning tree with distribution $\PP$.

Although this is a rigorous and precise description of the model, it is not very helpful when it comes to explicit computations. In fact, for this purpose, it is more convenient to describe the random spanning tree in terms of the positions of its vertical edges (rungs) and its missing horizontal edges. Before we introduce the somewhat technical notation, let us explain the concept.

\begin{figure}[ht]
\centerline{
\begin{tikzpicture}[scale=0.87]
\def\myrad{0.14cm}
\foreach \x in {0,2,3,4,5,6,7,8,9,11,12,13,15}
  \draw[color=blue, line width = 0.5mm] ({\x},1) -- ({\x+1},1);
\foreach \x in {0,1,2,3,5,6,7,8,9,10,11,12,13,14,15}
  \draw[color=blue, line width = 0.5mm] ({\x},0) -- ({\x+1},0);
\foreach \x in {3,8,14}
  \draw[color=blue, line width = 0.5mm] ({\x},0) -- ({\x},1);
\foreach \x in {0,...,16}
  \foreach \y in {0,1}
  \shadedraw ({\x},{\y}) circle(\myrad);
\draw (2.0, -0.2) node[anchor=north]{$H_{-1}$};
\draw (3.0, -0.2) node[anchor=north]{$V_{-1}$};
\draw (5.0, -0.2) node[anchor=north]{$H_0$};
\draw (7.0, -0.2) node[anchor=north]{$0$};
\draw (8.0, -0.2) node[anchor=north]{$V_0$};
\draw (11.0, -0.2) node[anchor=north]{$H_1$};
\draw (14.0, -0.2) node[anchor=north]{$V_1$};
\draw (15.0, -0.2) node[anchor=north]{$H_2$};
\draw[->,>=stealth',thick] (2,1.14) arc[radius=0.53, start angle=160, end angle=20];
\draw[->,>=stealth',thick] (3.08,1.14) arc[radius=1.46, start angle=130, end angle=50];
\draw[->,>=stealth',thick] (5.08,1.14) arc[radius=3.0, start angle=118, end angle=62];
\draw[->,>=stealth',thick] (8.08,1.14) arc[radius=3.0, start angle=118, end angle=62];
\draw[->,>=stealth',thick] (11.08,1.14) arc[radius=3.0, start angle=118, end angle=62];
\draw[->,>=stealth',thick] (14,1.14) arc[radius=0.53, start angle=160, end angle=20];
\draw (2.5, 1.5) node[anchor=south]{$F'_{-1}$};
\draw (4.0, 1.5) node[anchor=south]{$F_{-1}+1$};
\draw (6.5, 1.5) node[anchor=south]{$F'_{0} $};
\draw (9.5, 1.5) node[anchor=south]{$F_{0}+1$};
\draw (12.5, 1.5) node[anchor=south]{$F'_{1} $};
\draw (14.5, 1.5) node[anchor=south]{$F_{1}+1 $};
\end{tikzpicture}
}
\caption[b]{Construction of the random spanning tree $\CT$. Illustration of $A_{3,2,-1,1}$ (explained later).}
\label{F1.2}
\end{figure}
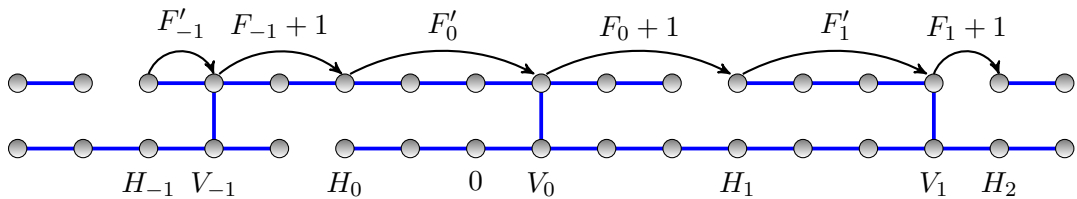

The tree is completely specified if we know the positions of the missing horizontal edges and the positions of the vertical edges (rungs) in the tree.
\begin{itemize}
\item Let $(H_n)_{n\in\Z}$ denote the horizontal positions of the right vertices of the missing rungs. Assume that the numeration is chosen such that
$$\ldots H_{-2}<H_{-1}<H_0\leq 0<H_1<H_2<\ldots$$
\item
Denote by $(W_n)$ the corresponding vertical positions of the missing edges.
\item Between any two horizontal positions $H_n$ and $H_{n+1}-1$ there is exactly one vertical edge in the tree. Denote the horizontal position of this edge by $V_n$. That is, $H_n\leq V_n<H_{n+1}$.
Note that $V_{-1}<0$ and $V_1>0$ but $V_0$ could have either sign or equal 0.
\end{itemize}

Roughly speaking, if we start from a rung at position $V_n$ there are a random number $F_n$ of positions to the right with both horizontal edges before the next horizontal edge is missing. That is, $V_n + F_n +1 = H_{n+1}$. Going right from $H_{n}$ there are a random number $F'_{n}$ of positions before the next rung at $V_{n}$. That is $H_n+F_{n}'=V_{n}$. Note that $$H_{n+1}-H_n=F_n+F'_n+1.$$

Following the work of Häggström \cite{Haggstrom1994} for the case $c=1$ and \cite{Klenke2017} for general $c>0$, the $(F_n)$ and $(F_n')$ and $(W_n)$ are independent random variables and
\begin{itemize}
\item $W_n$ takes the values $0$ and $1$ each with probability $1/2$
\item $F_n$ and $F_n'$, $n\neq 0$, are geometrically distributed $\gamma_{1-\alpha}$ with parameter $1-\alpha$ with $\alpha$ defined in \equ{E1.02}.
\end{itemize}
Here
\begin{equation}
\label{E1.02}
\alpha:=c+1-\sqrt{c^2+2c}\;\in\,(0,1),
\end{equation}
and the geometric distribution with parameter $a\in(0,1]$ is defined by
\begin{equation}\label{E1.03}
\gamma_{a}(k) = a\cdot (1-a)^k, \quad k=0,1,2,\ldots
\end{equation}
Note that $\alpha$ is a monotone decreasing function of $c$ and $\alpha\to0$ as $c\to\infty$ (and hence the $F_n$ and $F_n'$ tend to $0$) and $\alpha\to1$ as $c\to0$.

Clearly, $(H_n)_{n\in\Z}$ is a stationary renewal process and the renewal times $G_n:=H_{n+1}-H_n=F_{n}+F'_{n}+1$ have distribution $\gamma_{1-\alpha}\ast\gamma_{1-\alpha}\ast\delta_1$ for $n\neq 0$. For $n=0$, however, the gap $G_0$ is a size-biased pick of this distribution (waiting time paradox). That is,
\begin{equation}
\label{E1.04}
\begin{aligned}
\PP[G_0=k]&=\frac{k\,\PP[G_1=k]}{\EE[G_1]}\\[2mm]
&=\frac{{1-\alpha}}{1+\alpha}\,(1-\alpha)^2\,\alpha^{k-1}\,k^2,\quad k=1,2,3,\ldots.
\end{aligned}
\end{equation}
Roughly speaking, by symmetry, given $G_0$, both the position of the origin and the position of the rung at $V_0$ are uniformly distributed among the possible values $\{H_{0},\ldots,H_1-1\}$ and are independent. In other words, given $G_0$, the random variables $H_0$ and $F'_0$ are independent and
$$\PP[-H_0=j\given G_0=k]=\PP[F'_0=j\given G_0=k]=\frac1k\mfa j=0,\ldots,k-1.$$
Note that
$$V_0=H_0+F'_0$$
is the difference of two independent and uniformly distributed random variables $F'_0$ and $-H_0$ (given $G_0$).

Summing up, the random spanning tree $\CT$ can be described in terms of the independent random variables $(F_n)$, $(F'_n)$, $(W_n)$ and by $H_0$. The $F_n$ and $F'_n$, $n\neq 0$, are $\gamma_{1-\alpha}$ distributed while for $F_0$ and $F'_0$ a somewhat different distribution needs to be chosen. (Since we are interested in asymptotic properties only, we would not even need to know the precise  distributions of $F_0$ and $F'_0$.) Given these random variables, the positions of the rungs and the missing edges in $\CT$ are given by:

\begin{equation}
\label{E1.05}
H_n = H_0+\sum_{k=0}^{n-1}(F_k+F_k'+1)\mfalls n\geq 1,
\end{equation}
\begin{equation}
\label{E1.06}
H_n=H_0-\sum_{k=n}^{-1}(F_k+F_k'+1)\mfalls n\leq -1,
\end{equation}
and
\begin{equation}
\label{E1.07}
V_n=H_n+F'_n,\qquad n\in\Z.
\end{equation}
\par

\textbf{Random walk on the spanning tree}\par
We now define random walk on $\CT$ in the spirit of the random conductance model. Denote by $\P^\CT$ the probabilities for a fixed spanning tree $\CT$. Furthermore, we let
\begin{equation}
\label{E1.08}
\P=\int \P^\CT\,\PP[d\CT]
\end{equation}
denote the annealed distribution and $\E$ its expectation.

Fix a parameter $\beta\geq1$ and attach to each edge in $\CT$ a weight (conductance)
\begin{equation}
\label{E1.09}
C(z_n)=C(h_{i,n})=\beta^n,\qquad n\in\Z,\,i=0,1.
\end{equation}
For $v\in\CT$, we write the sum of the conductances of adjacent edges by
$$C(v):=\sum_{e:\,v\in e}C(e).$$
Note that $C(v)$ depends on $\CT$ but this dependence is suppressed in the notation.

The random walk $X=(X^{(1)},X^{(2)})$ on $\CT$ chooses among its neighboring edges with a probability proportional to the edge weight. That is,
$$\P^\CT[X_{n+1}=w\given X_n=v]=\frac{C(\{v,w\})}{C(v)}\quad\mbox{for }\{v,w\}\in\CT.$$
Note that $X^{(1)}\in\{0,1\}$ is the vertical position and $X^{(2)}\in\Z$ is the horizontal position of $X$.

\subsection{Main Results}
\label{S1.3}
For $\beta>1$, this random walk has a bias to the right and we will see that it is in fact transient to the right and that the asymptotic speed
$$v:=\lim_{n\to\infty} \frac{X^{(2)}_n}{n}$$
exists. Since $\CT$ is ergodic, the value of $v$ does not depend on $\CT$ and is a deterministic function of $\alpha$ and $\beta$. In this paper, we give an explicit formula for $v$ and we discuss how $v$ depends on $\beta$ and on $\alpha$.
In particular, we see that $v$ is strictly positive if and only if $\beta < \beta_c^{(1)}:=1/\alpha$, and that $v$ is a unimodal function of $\beta$.
For random walk on the full ladder graph, the speed is a monotone function of $\beta$. However, in the spanning tree, right of the vertical edges there are dead ends of varying sizes where the random walk can spend large amounts of time if $\beta$ is large. Hence it can be expected that there exists a critical value $\beta_c^{(1)}=\beta_c^{(1)}(\alpha)$ such that $v>0$ for $\beta\in\big(1,\beta_c^{(1)}\big)$ and $v=0$ for $\beta\geq \beta_c^{(1)}$.
For $\beta < 1/\alpha$, let
\begin{equation}
\label{E1.10}
\begin{aligned}
s_+&:=\frac{1-\alpha}{1-\alpha\beta}\\
s_-&:=\frac{1-\alpha}{1-\alpha/\beta}
\end{aligned}
\end{equation}
and
\begin{equation}
\label{E1.11}
C=\frac{2\beta}{\beta-1}-\frac{1-s_-/\beta}{1-s_-^2/\beta}=\frac{2\beta}{\beta-1}-\frac{\beta-\alpha}{\beta-\alpha^2}.  \end{equation}
(For a more intuitive description of these quantities, see \eqref{E2.07}, \eqref{E2.08} and \eqref{E2.09} in the proof).
\begin{theorem}[Asymptotic Speed]
\label{T1}
Let $\beta_c^{(1)}:=1/\alpha$. For $\beta\geq\beta_c^{(1)}$, we have $v=v(\beta)=0$. For $\beta\in\big(1,\beta_c^{(1)}\big)$, we have $v(\beta)>0$ and the value of $v$ is given by
\begin{equation}
\label{E1.12}
\frac1{v(\beta)}=\frac{\beta+1}{\beta-1}+\frac{1-\alpha}{1+\alpha}\left(\frac{\beta+3}{2(\beta-1)}+\frac{\beta(\beta+1)}{(\beta-1)^2}(s_+-s_-)+C\,s_-\right).
\end{equation}
\end{theorem}

Note that $\frac{\beta-1}{\beta+1}$ is the asymptotic speed of random walk on $\Z$ with a bias $\beta$ to the right. The remaining terms on the r.h.s.{} of \equ{E1.12} describe the slowdown due to (1) traps and (2) the lengthening of the path due to the need to pass vertical edges. Note that $C > 0$ and $(s_+-s_-) > 0$.

Clearly, $v(\beta) = 0$ for $\beta \in \{1, 1/\alpha\}$. To see that for each value of $\alpha \in (0,1)$, $\beta\mapsto v(\beta)$ is a unimodal function, it suffices to show that $\beta \mapsto 1/v(\beta)$ is convex on $(1, 1/\alpha)$, and the readers can convince themselves from \eqref{E1.12} that this is the case since $\beta \mapsto 1/v(\beta)$ is a sum of convex functions.
\begin{figure}[ht]
\centerline{\includegraphics[width=12cm]{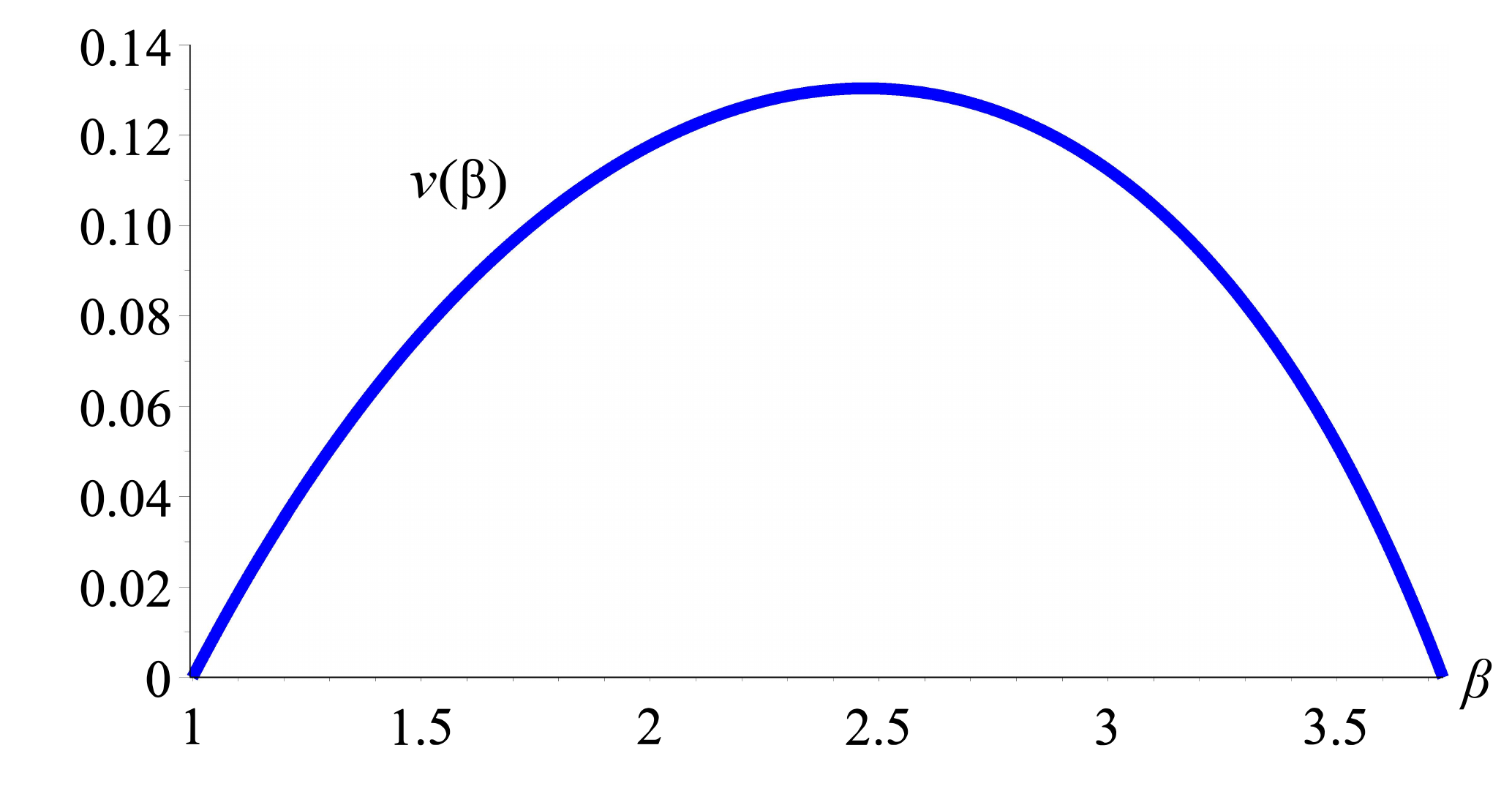}}
\caption[b]{A sketch of $\beta\mapsto v(\beta)$ for $\alpha=2-\sqrt{3}=0.2679\ldots$ ($c=1$).}
\label{F1.3}
\end{figure}

The explicit formula for the speed allows to investigate the dependence on the parameters $\beta$ and $\alpha$.
Taking the limit of $v$ as $\alpha\to0$ (which amounts to $c\to\infty$) in \equ{E1.12}, we get
\begin{equation}
\label{E1.13}
\lim_{\alpha\downarrow0}v=\frac{2(\beta-1)}{5\beta+7}.
\end{equation}
Note that this corresponds to the speed of a biased RW on a uniform spanning tree with all vertical edges.
The uniform spanning tree can be chosen as follows: for each pair of horizontal edges $h_{0,k}$, $h_{1,k}$, a fair coin flip decides which one is retained. For this case, the formula \equ{E1.13} could be derived directly by a straightforward (but not short) Markov chain argument.

Does the speed increase as $\alpha$ increases, since there are less vertical edges to slow down the random walk? Or does increasing $\alpha$ mean that the traps get larger and the speed decreases? The latter effect should be stronger for large $\beta$ and in fact we have
\begin{equation}
\label{E1.14}
\lim_{\alpha\downarrow0}\frac{\partial v}{\partial \alpha}=\frac{4(\beta-1)(\beta+1)(3-\beta)}{(5\beta+7)^2}
\end{equation}
which is positive if $\beta<3$ and negative if $\beta>3$. Hence, for fixed $\beta$, the value of the speed can either increase or decrease in the neighborhood of $\alpha = 0$.
\begin{figure}[ht]
\centerline{\includegraphics[width=7cm]{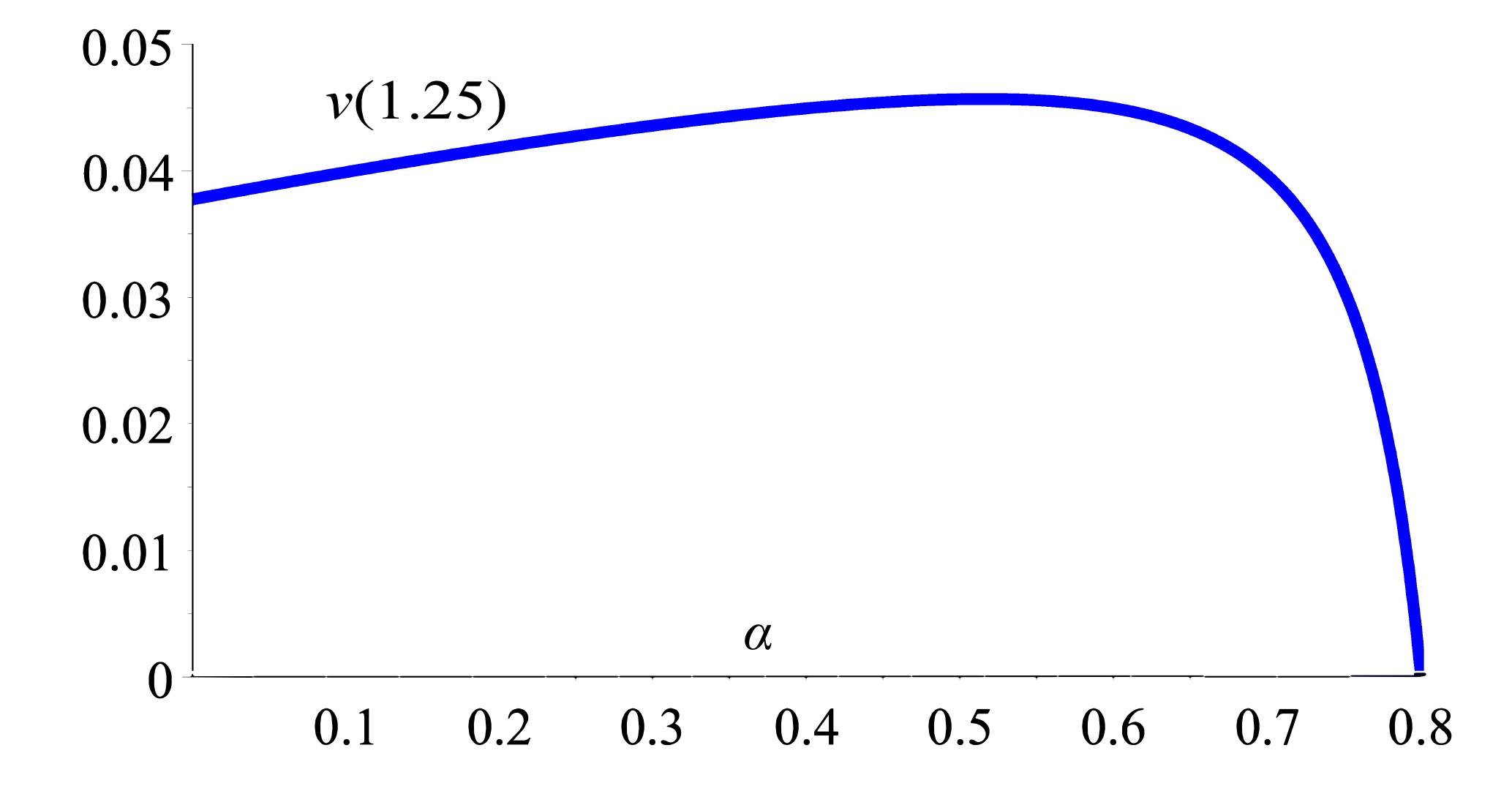}\quad\includegraphics[width=7cm]{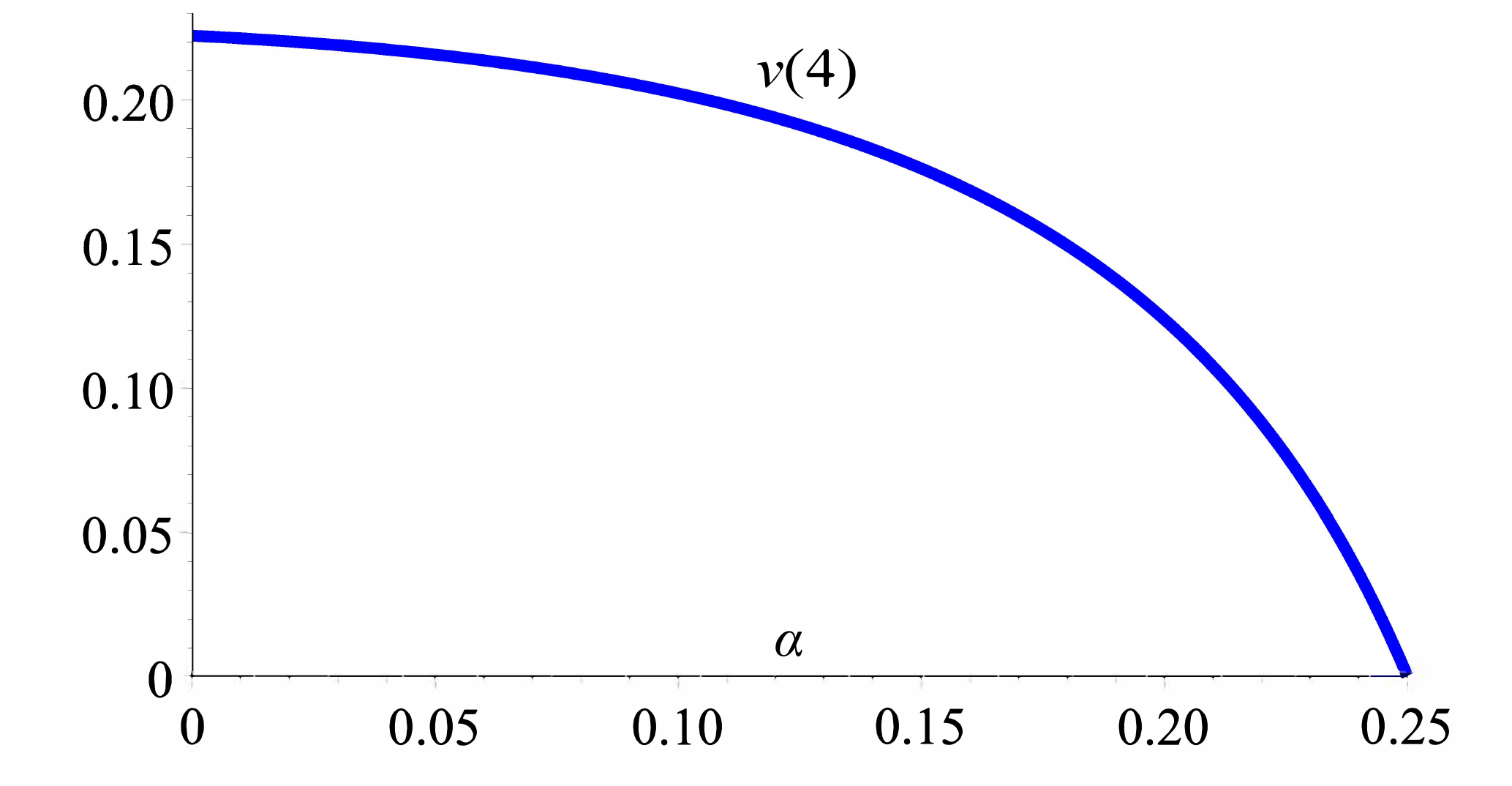}}
\caption[b]{A sketch of $v(1.25)$ (left) and $v(4)$ (right) as a function of $\alpha$.}
\label{F1.4}
\end{figure}

The next goal is to establish a central limit theorem in the ballistic regime, that is, in the regime where $v(\beta)>0$. As we will need second moments, we have to restrict the range of $\beta$ further.
Assume that $\beta\in\big(1,1/\sqrt{\alpha}\big)$. Note that $v(\beta)>0$ and that
$$\varrho:=-\frac{\log(\alpha)}{\log(\beta)}>2.$$

By Theorem 1.1 of \cite{GantertKlenke2022}, the time the random walk spends in a trap has tails with moments of all orders smaller than $\varrho$ but no moments larger than or equal to $\varrho$. Hence, the critical value $\beta_c^{(2)}$ for the existence of second moments is
\begin{equation}
\label{E1.15}
\beta_c^{(2)}:=\frac{1}{\sqrt{\alpha}}.
\end{equation}

For $\beta\in\big(1,\beta_c^{(2)}\big)$, second moments exist and this indicates that a central limit theorem should hold in this regime. Denote by $\CN_{0,1}$ the standard normal distribution.

\begin{theorem}
\label{T2}
Assume that $\beta\in\big(1,\beta_c^{(2)}\big)$. Then there exists a $\varsigma^2\in(0,\infty)$ such that the annealed laws converge to a standard normal distribution, i.e.
\begin{equation}
\label{E1.16}
\CL_\P\left[\frac{X^{(2)}_n-v(\beta) n}{\sqrt{\varsigma^2n}}\right]\limn \CN_{0,1}.
\end{equation}
\end{theorem}

Note that
\begin{equation}
\label{E1.17}
\beta_c^{(2)}<\beta_{\mathrm{max}}
\end{equation}
where $\beta_{\mathrm{max}}$ is the value $\beta$ for which $v(\beta)$ is maximal. In fact, an explicit calculation gives
$$\frac{\partial v}{\partial \beta}\restbig{\beta=\beta_c^{(2)}}=\frac{p(\beta)}{q(\beta)}$$
with
$$\begin{aligned}p(\beta)&=
4 \Big( 5\,{\beta}^{14}+20\,{\beta}^{13}+51\,{\beta}^{12}+94\,{\beta}^{11}+
141\,{\beta}^{10}+180\,{\beta}^{9}+203\,{\beta}^{8}+196\,{\beta}^{7}\\
&\qquad +164\,{\beta}^{6}+118
\,{\beta}^{5}+72\,{\beta}^{4}+34\,{\beta}^{3}+14\,{\beta}^{2}+6\,\beta+2 \Big)  \left( {
\beta}^{2}+1 \right)\\
q(\beta)&= \left( 5\,{\beta}^{9}+19\,{\beta}^{8}+36\,{\beta}^{7}+51\,{\beta}^
{6}+57\,{\beta}^{5}+47\,{\beta}^{4}+33\,{\beta}^{3}+22\,{\beta}^{2}+9\,\beta+1 \right) ^{2
}.\end{aligned}$$
As all coefficients are positive, we have
$$\frac{\partial v}{\partial \beta}\restbig{\beta=\beta_c^{(2)}}>0.$$

In the case $\beta\in\big(\beta_c^{(2)},\beta_c^{(1)}\big)$, we are still in the ballistic regime but second moments of the time spent in traps fail to exist. In fact, the $p$th moment exists if and only if $p<\varrho$. See
\cite{GantertKlenke2022}. Hence we might ask if a proper rescaling yields convergence to a stable law. This, however, cannot be expected to hold since the time that the random walk spends in a random trap does not have regularly varying tails. See \cite{GantertKlenke2022} for a detailed discussion.

Compare the situation to the random comb model (with exponential tails) (see \cite[Fig. 3 and Fig. 9 (b)]{KotakBarma}): There are two critical values $g_1>g_2>0$ for the drift $g$ such that the following happens.
\begin{itemize}
\item For small drift $0<g<g_2$, the random comb is in the Normal Transport regime (NT). That is, the speed is positive and the second moments are finite. The phase transition at $g_2$ is of second order, that is, the second moments diverge. Also for our model, the second moments of the time spent in traps diverges as $\beta\uparrow\beta_c^{(2)}$ as can be read off from the explicit formula of the tails given in \cite{GantertKlenke2022}.
\item For $g_2 < g < g_1$, the random comb is in the Anomalous Fluctuation regime (AF). The speed is positive but the second moments are infinite. The phase transition at $g_1$ is of second order, that is, the speed converges to $0$ as $g\to g_1$. Also for our random walk on the random spanning tree, the speed decreases to $0$ as $\beta\uparrow\beta_c^{(1)}$.
\item For $g_1<g$, the random comb is in the Vanishing Velocity regime (VV) where the speed is zero.
\end{itemize}
Also in the random comb model, the speed is an increasing function of the drift $g$ in the NT regime (for exponential tails), as can be seen in \cite[Fig. 9 (b)]{KotakBarma}, and the speed is maximal at some $g_{\mathrm{max}}>g_2$ just as we have shown for our model in \equ{E1.17}.

We come to the final goal of this paper. In the unbiased case $\beta=1$, all moments of the time spent in traps exist and a central limit theorem should hold. Moreover, the variance $\sigma^2$ should be given by the Einstein relation
\begin{equation}
\label{E1.18}
\sigma^2=2\frac{\partial v}{\partial \beta}\restbig{\beta=1}=\frac{1+\alpha}{3+\alpha}
\end{equation}
where the second equality can easily checked by an explicit computation.
We show that this is indeed true and that a quenched invariance principle holds true with this value of $\sigma^2$.
\begin{theorem}[Quenched Functional Central Limit Theorem]
\label{T3}
Let $\beta=1$ and let $\sigma^2$ be given by \equ{E1.18}. The process
$$\left(\frac{1}{\sqrt{\sigma^2\,n}}X^{(2)}_{\lfloor tn\rfloor}\right)_{t\geq0}$$ converges in $\P^\CT$-distribution in the Skorohod space $D_{[0,\infty)}$ to a standard Brownian motion for $\PP[d\CT]$ almost all spanning trees $\CT$.
\end{theorem}

\begin{figure}[ht]
\centerline{\includegraphics[height=7cm]{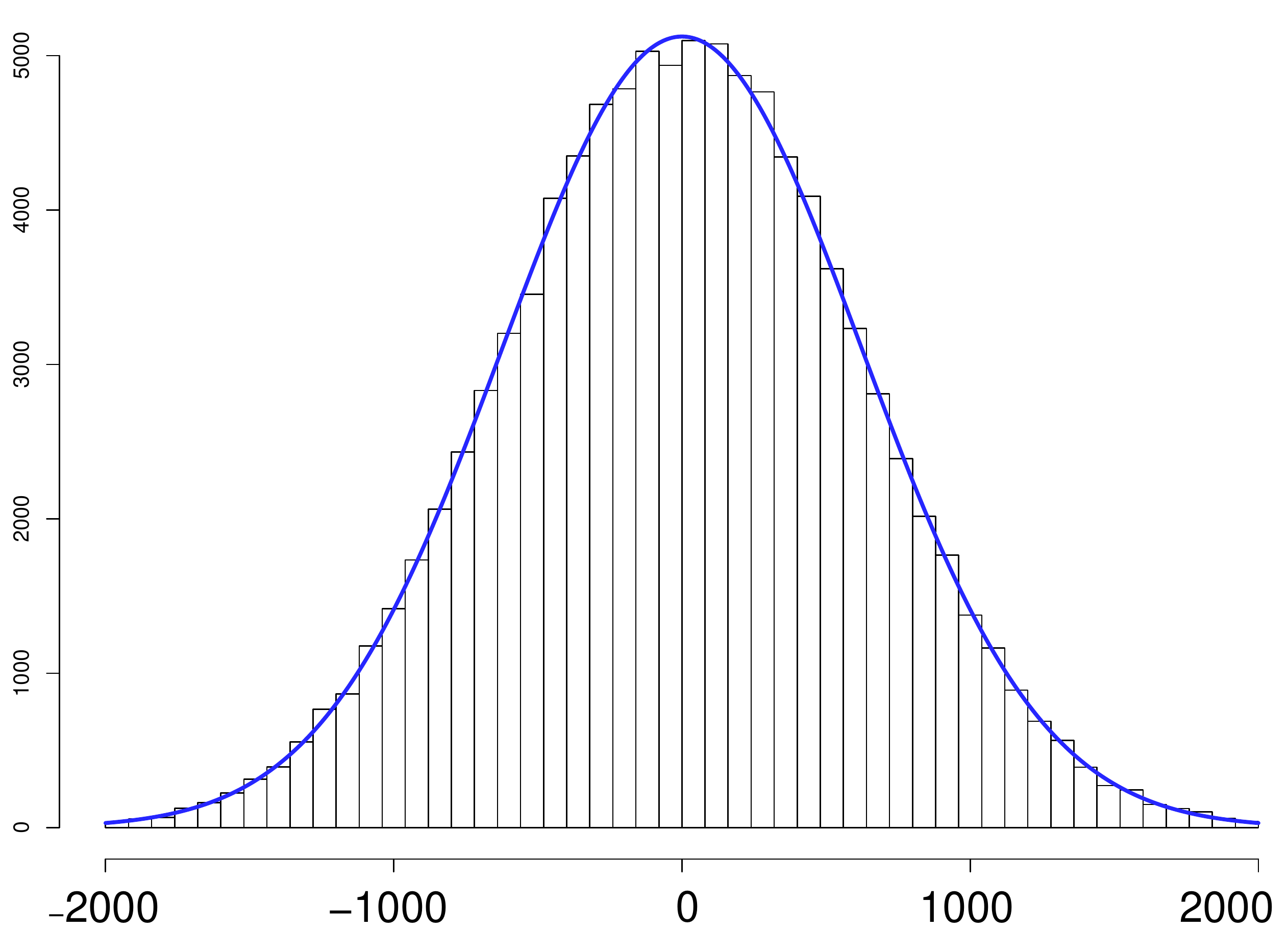}}
\caption[b]{Uniform spanning tree ($c=1$, $\alpha=2-\sqrt{3}$). No drift: $\beta=1$. The histogram shows the endpoints as a result of 100\,000 simulations of 1\,000\,000 steps of RWRE. The curve shows the (scaled) density of the normal distribution with variance $10\,000\times\frac{1+\alpha}{3+\alpha}$ as suggested by the CLT.}
\label{F1.5}
\end{figure}

\subsection{Organization of the paper}
\label{S1.4}
The rest of the paper is organized as follows.
In Section~\ref{S2}, we study the random conductances in some more detail. We also decompose the spanning tree into building blocks and compute the expected times the random walks spends in these blocks. We put things together to get the explicit formula for the speed and prove Theorem~\ref{T1}.

In Section~\ref{S3}, we use a regeneration time argument to infer the central limit theorem in the ballistic regime.

In Section~\ref{S4}, we prove the Einstein relation (Theorem~\ref{T3}) by computing second moments.
\Section{Conductances Model and Proof of Theorem~\ref{T1}}
\label{S2}
\subsection{Some more considerations on the spanning tree}
\label{S2.1}
In the spanning tree $\CT$, there is a unique ray (self avoiding path) from the left to the right. We denote the ray by $\Ray$. We can enumerate the ray by the positive integers following it from $(0,0)$ [or (1,0) if (0,0) is not in the ray] to the right and by the negative integers going to the left. Let $\phi(i)=(\phi^{(1)}(i),\phi^{(2)}(i))$, $i\in\Z$, be this enumeration.

The basic idea is to decompose the random walk $X$ into a random walk $Y$ that makes steps only on $\Ray$ with edge weights given by \equ{E1.09}. That is, if $Y_n$ is in the ray and also the position one step to the right, that is $\phi(\phi^{-1}(Y_n)+1)=Y_n+(0,1)$, is in the ray, then  $\P^\CT[Y_{n+1}=\phi(\phi^{-1}(Y_n)+1)\Given Y_n]=\beta/(1+\beta)$ and
$\P^\CT[Y_{n+1}=\phi(\phi^{-1}(Y_n)-1)\Given Y_n]=1/(1+\beta)$. Otherwise these probabilities are both $1/2$. Now assume that we attach random holding times to $Y$ that model the times that $X$ spends in the traps splitting from the ray. In most cases, of course, these holding times are simply 1 since there are no traps.

There are three kinds of traps:
\begin{itemize}
\item[(a)] A horizontal edge splits to the right of the ray: The ray makes a step down (or up) and the trap consists of a number, say $k$, of horizontal edges to the right of the turning point. See Figure \ref{F2.1}.
\item[(b)] A horizontal edge splits to the left of the ray: The ray has just made a step down (or up) and now turns to the right again. The trap consists of a number, say $l$, of horizontal edges to the left of the turning point. See Figure \ref{F2.2}.
\item[(c)] A vertical edge splits from the ray. The ray has just made a step to the right and will make another step to the right. A vertical edge either splits to the top or the bottom of the ray. At the other end of this vertical edge, there are $k$ horizontal edges to the right and $l$ horizontal edges to the left. See Figure \ref{F2.3}.
\end{itemize}

\begin{figure}[ht]
\centerline{
\begin{tikzpicture}[scale=1.5]
\def\myrad{0.14cm}
\draw[color=blue, line width = 0.5mm] (0,1) -- (3,1);
\draw[color=red, line width = 0.5mm] (3,1) -- (6,1);
\draw (7,1) -- (8,1);
\draw (1,0) -- (3,0);
\draw[color=blue, line width = 0.5mm] (3,0) -- (8,0);
\draw (7,0) -- (8,0);
\draw[color=blue, line width = 0.5mm] (3,0) -- (3,1);
\shadedraw(0,1) circle(\myrad) ;
\shadedraw (1,1) circle(\myrad);
\shadedraw (2,1) circle(\myrad);
\draw[fill, color=teal]  (3,1) circle(\myrad);
\draw  (3,1) circle(\myrad);
\shadedraw (4,1) circle(\myrad);
\shadedraw (5,1) circle(\myrad);
\shadedraw (6,1) circle(\myrad);
\shadedraw (7,1) circle(\myrad);
\shadedraw (8,1) circle(\myrad);
\shadedraw (0,0) circle(\myrad) ;
\shadedraw (1,0) circle(\myrad);
\shadedraw (2,0) circle(\myrad);
\shadedraw (3,0) circle(\myrad);
\shadedraw (4,0) circle(\myrad);
\shadedraw (5,0) circle(\myrad);
\shadedraw (6,0) circle(\myrad);
\shadedraw (7,0) circle(\myrad);
\shadedraw (8,0) circle(\myrad);
\draw (1.5,-0.0) node[anchor=north]{$\beta^{-1}$};
\draw (1.5, 1.6) node[anchor=north]{$\beta^{-1}$};
\draw (2.5,-0.0) node[anchor=north]{$1$};
\draw (2.5, 1.6) node[anchor=north]{$1$};
\draw (2.8, 0.8) node[anchor=north]{$1$};
\draw (3.5,-0.0) node[anchor=north]{$\beta$};
\draw (3.5, 1.6) node[anchor=north]{$\beta$};
\draw (4.5,-0.0) node[anchor=north]{$\beta^2$};
\draw (4.5, 1.6) node[anchor=north]{$\beta^2$};
\draw (5.5,-0.0) node[anchor=north]{$\beta^3 $};
\draw (5.5, 1.6) node[anchor=north]{$\beta^3$};
\end{tikzpicture}
}\caption{Trap (a) (in red) with  $k=3$. Initial point in green, ray in blue.}
\label{F2.1}
\end{figure}
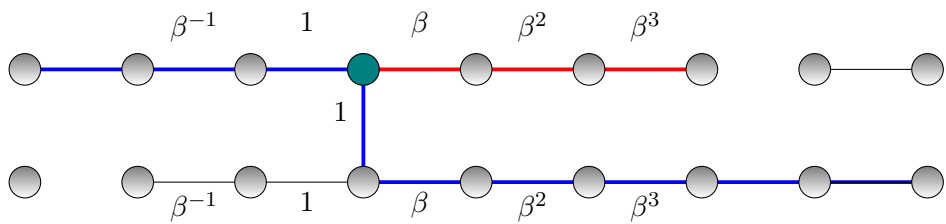

\begin{figure}[ht]
\centerline{
\begin{tikzpicture}[scale=1.5]
\def\myrad{0.14cm}
\draw[color=blue, line width = 0.5mm] (0,1) -- (3,1);
\draw (3,1) -- (6,1);
\draw (7,1) -- (8,1);
\draw[color=red, line width = 0.5mm] (1,0) -- (3,0);
\draw[color=blue, line width = 0.5mm] (3,0) -- (8,0);
\draw (7,0) -- (8,0);
\draw[color=blue, line width = 0.5mm] (3,0) -- (3,1);
\shadedraw(0,1) circle(\myrad) ;
\shadedraw (1,1) circle(\myrad);
\shadedraw (2,1) circle(\myrad);
\shadedraw (3,1) circle(\myrad);
\shadedraw (4,1) circle(\myrad);
\shadedraw (5,1) circle(\myrad);
\shadedraw (6,1) circle(\myrad);
\shadedraw (7,1) circle(\myrad);
\shadedraw (8,1) circle(\myrad);
\shadedraw (0,0) circle(\myrad) ;
\shadedraw (1,0) circle(\myrad);
\shadedraw (2,0) circle(\myrad);
\draw[fill, color=teal] (3,0) circle(\myrad);
\draw  (3,0) circle(\myrad);
\shadedraw (4,0) circle(\myrad);
\shadedraw (5,0) circle(\myrad);
\shadedraw (6,0) circle(\myrad);
\shadedraw (7,0) circle(\myrad);
\shadedraw (8,0) circle(\myrad);
\draw (1.5,-0.0) node[anchor=north]{$\beta^{-1}$};
\draw (1.5, 1.6) node[anchor=north]{$\beta^{-1}$};
\draw (2.5,-0.0) node[anchor=north]{$1$};
\draw (2.5, 1.6) node[anchor=north]{$1$};
\draw (2.8, 0.8) node[anchor=north]{$1$};
\draw (3.5,-0.0) node[anchor=north]{$\beta$};
\draw (3.5, 1.6) node[anchor=north]{$\beta$};
\draw (4.5,-0.0) node[anchor=north]{$\beta^2$};
\draw (4.5, 1.6) node[anchor=north]{$\beta^2$};
\draw (5.5,-0.0) node[anchor=north]{$\beta^3 $};
\draw (5.5, 1.6) node[anchor=north]{$\beta^3$};
\end{tikzpicture}
}\caption{Trap (b) (in red) with  $l=2$. Initial point in green, ray in blue.}
\label{F2.2}
\end{figure}
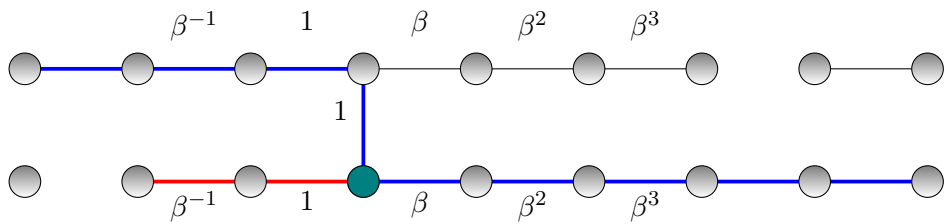
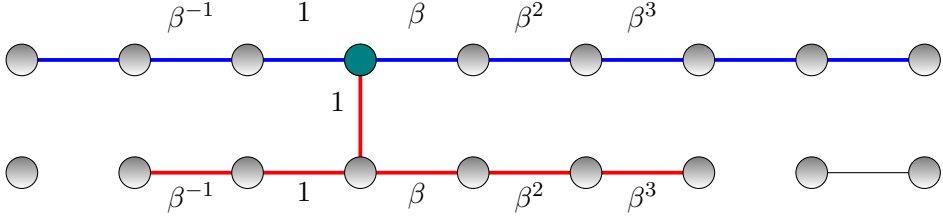
\begin{figure}[ht]
\centerline{
\begin{tikzpicture}[scale=1.5]
\def\myrad{0.14cm}
\draw[color=blue, line width = 0.5mm] (0,1) -- (8,1);
\draw[color=red, line width = 0.5mm] (1,0) -- (6,0);
\draw (7,0) -- (8,0);
\draw[color=red, line width = 0.5mm] (3,0) -- (3,1);
\shadedraw(0,1) circle(\myrad) ;
\shadedraw (1,1) circle(\myrad);
\shadedraw (2,1) circle(\myrad);
\draw[fill, color=teal]  (3,1) circle(\myrad);
\draw  (3,1) circle(\myrad);
\shadedraw (4,1) circle(\myrad);
\shadedraw (5,1) circle(\myrad);
\shadedraw (6,1) circle(\myrad);
\shadedraw (7,1) circle(\myrad);
\shadedraw (8,1) circle(\myrad);
\shadedraw (0,0) circle(\myrad) ;
\shadedraw (1,0) circle(\myrad);
\shadedraw (2,0) circle(\myrad);
\shadedraw (3,0) circle(\myrad);
\shadedraw (4,0) circle(\myrad);
\shadedraw (5,0) circle(\myrad);
\shadedraw (6,0) circle(\myrad);
\shadedraw (7,0) circle(\myrad);
\shadedraw (8,0) circle(\myrad);
\draw (1.5,-0.0) node[anchor=north]{$\beta^{-1}$};
\draw (1.5, 1.6) node[anchor=north]{$\beta^{-1}$};
\draw (2.5,-0.0) node[anchor=north]{$1$};
\draw (2.5, 1.6) node[anchor=north]{$1$};
\draw (2.8, 0.8) node[anchor=north]{$1$};
\draw (3.5,-0.0) node[anchor=north]{$\beta$};
\draw (3.5, 1.6) node[anchor=north]{$\beta$};
\draw (4.5,-0.0) node[anchor=north]{$\beta^2$};
\draw (4.5, 1.6) node[anchor=north]{$\beta^2$};
\draw (5.5,-0.0) node[anchor=north]{$\beta^3 $};
\draw (5.5, 1.6) node[anchor=north]{$\beta^3$};
\end{tikzpicture}
}\caption{Trap (c) with $k=3$ and $l=2$. Initial point in green, ray in blue.}
\label{F2.3}
\end{figure}

It is tempting to follow a simple (but wrong) argument to compute the asymptotic speed of $X$: For a given point on the ray, compute the probability $p$ that a trap starts at this point and compute the average time $\E[T]$, the walk spends in a trap. Then take the speed $v^Y$ of the random walk $Y$ on the ray and divide it by $1+p\E[T]$. What is wrong with the argument is that the expected numbers of visits to a given point on the ray vary in a subtle way depending on the ray. Hence, we will have to argue more carefully to prove Theorem~\ref{T1}. However, there is a nice simplification of our model where this approach works and as a warm-up we present this here.

Consider the spanning tree $\CT^u$ that is defined just as $\CT$ but the ray is $\{1\}\times\Z$. In other words, we would have $W_n=0$ for all $n$. In this case, there are only traps of type (c). Also, the random walk $Y^u$ on the ray is simply a random walk on $\Z$ with a drift $\frac{\beta-1}{\beta+1}$. Since the gaps between vertical edges have distribution $\gamma_{1-\alpha}\ast\gamma_{1-\alpha}\ast\delta_1$ they have expectation $\frac{1+\alpha}{1-\alpha}$. Hence, the probability for a given vertical edge to be in $\CT^u$
is
\begin{equation}
\label{E2.01}
p=\frac{1-\alpha}{1+\alpha}.
\end{equation}

Assume that the trap starts one step
below (or above in the general case $\CT$)
the initial point of the random walk and then splits into $l$ edges to the left and $k$ edges to the right. Let $T_{k,l}$ be the time, the random walk on $\CT^u$ spends in the trap before it makes the first step on the ray.
By Lemma~\ref{L2.03} below, we have
$$
\E^{\CT^u}[T_{k,l}]=
\frac{2}{1+\beta}\left(1+\frac{\beta}{\beta-1}\big(\beta^k-\beta^{-l}\big)\right).
$$
We still have to average over $k$ and $l$ to get, for $\beta < 1/\alpha$,
\begin{equation}
\label{E2.02}
\E[T]=\sum_{k,l=0}^\infty (1-\alpha)^2\,\alpha^{k+l}\,\E^\CT[T_{k,l}]=
\frac{2}{\beta+1}\left(1+\frac{\beta}{\beta-1}(s_+-s_-)\right)
\end{equation}
with $s_+$ and $s_-$ from \equ{E1.10}.
Summing up, we get that the biased random walk on $\CT^u$ has, for $\beta < 1/\alpha$, asymptotic speed
\begin{equation}
\label{E2.03}
v^u:=\frac{\beta-1}{\beta+1}\frac{1}{1+p\E[T]}=\frac{\beta-1}{\beta+1}\left(1+\frac{1-\alpha}{1+\alpha}\frac{2}{\beta+1}\left(1+\frac{\beta}{\beta-1}(s_+-s_-)\right)
\right)^{-1}.
\end{equation}

As indicated above, the computation of the speed $v$ of random walk on $\CT$ requires more care. We prepare for this in the next section by considering hitting times in the random conductance model.
\subsection{Conductance method and times in traps}
\label{S2.2}
If the random walker on the random spanning tree enters a dead end of the tree (a trap) it takes a random amount of time to get back to the ray of the tree. The average amount of time spent in a trap is the key quantity for computing the ballistic speed of the random walk. In this section, we first present the (well-known) formula that computes the average time to exit the trap in terms of the sum of edge weights \equ{E2.04}. Then we apply this formula to the three prototypes of traps: (a) a dead end to the right, (b) a dead end to the left and (c) a rung with dead ends to both sides.

Let $G=(V, E)$ be a connected graph and assume that we are given edge weights $C(x,y)=C(y,x)>0$ for all $\{x,y\}\in E$ and $C(x,y)=0$ otherwise. Let $C(x)=\sum_{y}C(x,y)$ and let
$$p(x,y)=\frac{C(x,y)}{C(x)},\quad x,y \in V$$
 be the transition probabilities of a reversible Markov chain $X$ on $V$. Assume that
$$C(V):=\sum_{y\in V}C(y)<\infty.$$ It is well known that the unique invariant distribution of this Markov chain is given by $\pi(x)= C(x)/C(V)$ for all $x\in V$. Let
$$\tau_x:=\inf\{n\geq1:X_n=x\}.$$ It is well known (see, e.g., \cite[Theorem 17.52]{Klenke2020e}) that
\begin{equation}
\label{E2.04}
\E_x[\tau_x]=\frac{1}{\pi(x)}=\frac{C(V)}{C(x)}.
\end{equation}

Now we assume that there is a point $x\in V$ with only one neighboring point $y$. Then
\begin{equation}
\label{E2.05}
\E_y[\tau_x]=\E_x[\tau_x]-1=\frac{C(V)}{C(x,y)}-1=2\,\frac{\sum_{e\in E}C(e)}{C(x,y)}-1.
\end{equation}

Quite similarly, assume that there are two points $x_1,x_2\in V$ each of which has only $y$ as a neighbor. Let $\tau_{\{x_1,x_2\}}$ be the first hitting time of $\{x_1,x_2\}$. By identifying the two points and giving the edge to $y$ the weight $C(x_1,y)+C(x_2,y)$, we get from \equ{E2.05}
\begin{equation}
\label{E2.06}
\E_y[\tau_{\{x_1,x_2\}}]=2\,\frac{\sum_{e\in E}C(e)}{C(x_1,y)+C(x_2,y)}-1.
\end{equation}
We will need \equ{E2.06} when we compute the expected time that the random walk $X$ on $\CT$ spends in a trap before it makes one step on the ray from $y$ to the left (say $x_1$) or the right (say $x_2$).

\begin{lemma}[Trap of type (a)]
\label{L2.01}
Assume that for a given tree $\CT$, the trap starts one step right of the initial point of the random walk and that the trap has $k\in\N$ edges. Let $T_{k}$ be the time the random walk spends in the trap before it makes the first step on the ray (not counting this first step on the ray). Then
$$\E^\CT[T_k]=\cases{
\beta\frac{\beta^k-1}{\beta-1},&\mfalls \beta\neq1,\\[2mm]
k,&\mfalls \beta=1.
}
$$
\end{lemma}
\begin{proof}
We apply \equ{E2.06} with $y$ the point on the ray where the trap splits off, $x_1$ and $x_2$ the two neighbouring points of $y$ on the ray. The graph $V$ consists of the points in the trap, $y$, $x_1$ and $x_2$. The sum of edge weights is
$$\sum_{e\in E}C(e)\;=\;2+\beta+\ldots+\beta^k.$$
Hence
$$\E^\CT[T_k]\;=2\;\frac{2+\beta+\ldots+\beta^k}{2}-2\;=\;\beta+\ldots+\beta^k.$$
This, however, is the assertion.
\end{proof}

\begin{lemma}[Trap of type (b)]
\label{L2.02}
Assume that for a given $\CT$, the trap starts one step left of the initial point of the random walk and that the trap has $l\in\N$ edges. Let $T_{l}$ be the time the random walk spends in the trap before it makes the first step on the ray (not counting this first step on the ray). Then
$$\E^\CT[T_l]=\cases{
\frac{2\beta}{\beta+1}\frac{1-\beta^{-l}}{\beta-1},&\mfalls \beta\neq1,\\[2mm]
l,&\mfalls \beta=1.
}
$$
\end{lemma}
\begin{proof}
We argue as in the proof of Lemma~\ref{L2.01}. Note that here the two edges on the ray have weights $1$ and $\beta$, respectively. Hence
$$\E^\CT[T_l]\;=2\;\frac{1+\beta+\beta^{0}+\beta^{-1}+\ldots+\beta^{1-l}}{\beta+1}-2
\;=\;2\,\frac{1+\ldots+\beta^{1-l}}{\beta+1}.$$
\end{proof}
\begin{lemma}[Trap of type (c)]
\label{L2.03}
Assume that for a given $\CT$, the trap starts one step above (or below) the initial point of the random walk and then splits into $l$ edges to the left and $k$ edges to the right. Let $T_{k,l}$ be the time the random walk spends in the trap before it makes the first step on the ray (not counting this first step on the ray).
Then
$$\begin{aligned}
\E^\CT[T_{k,l}]=\cases{
\frac{2}{1+\beta}\left(1+\frac{\beta}{\beta-1}\big(\beta^k-\beta^{-l}\big)\right),&\mfalls\beta\neq1,\\
k+l+1,&\mfalls\beta=1.}
\end{aligned}
$$
\end{lemma}
\begin{proof}
We argue as in the proof of Lemma~\ref{L2.02} to get
$$\E^\CT[T_{k,l}]\;=2\;\frac{1+\beta+1+\beta^{1-l}+\ldots+\beta^{k}}{\beta+1}-2
\;=\;2\,\frac{1+\beta^{1-l}+\ldots+\beta^{k}}{\beta+1}.$$
\end{proof}

\subsection{Proof of Theorem~\ref{T1}}
\label{S2.3}
This conductance approach is perfectly suited to study random walk on spanning trees. In fact, the speed of random walk on the random spanning tree can be derived from the average time it takes to make a step to the right on the infinite ray. From the considerations of the previous section it is clear that we need to compute the sum of the edge weights of all edges in $\CT$ that are left to walker (i. e., edges that can be reached without making a step to the right on the ray). The random spanning tree is made of i.i.d.{} building blocks that consist of the subgraphs between two missing horizontal edges. So we need the average weights of i.i.d.{} building blocks plus the edge weights in the block the walker is currently in. For this block, we need the exact shape of the block and the exact position of the walker before we take averages.

 Recall that $\CT$ is the random spanning tree with enumeration $\phi(n)$, $n\in\Z$ of its ray. Assume that we are given conductances $C(h_{i,k})=C(z_k)=\beta^{k}$ for all $h_{i,k},z_k\in\CT$. Let $X$ be the random walk on $\CT$ with these conductances. Let $\tau_0=0$ and for $k\in\N$, let
$$\tau_k:=\inf\big\{n\geq0:\phi^{(2)}(Y_n)=k\big\}.$$
Denote by $\E^\CT_0$ the quenched expectation for the random walk $X$ started at $\phi(0)$.
Since $\CT$ is stationary and ergodic, $\phi^{(2)}(X)$ has a deterministic asymptotic speed $v=v(\beta)$ given by
$$v=\frac{1}{\EE\E^\CT_0[\tau_1]}.$$

We will later condition on the position of the first horizontal edge missing to the left, count the edge weights between it and the origin by hand and average over the edge weights to the left of it. By translation invariance, it is enough to condition on
$$B:=\{H_0=0\}=\Big\{\#(\CT\cap\{h_{0,0},h_{1,0}\})=1\Big\}.$$
Denote by $\CT^-\subset\CT$ the set of edges in $\CT$ with at least one endpoint strictly to the left of $0$:
$$\CT^-=\CT\cap\Big\{h_{i,k},\,z_{k-1}:i=0,1,\,k\leq0\Big\}.$$ We now compute
\begin{equation}
\label{E2.07}
\begin{aligned}
C:&=\EE\left[\sum_{e\in\CT^-}C(e)\ggiven B\right]\\
&=\sum_{i=0}^1\sum_{k=-\infty}^0\PP\big[h_{i,k}\in\CT\Given B\big]\,\beta^k
+\sum_{k=-\infty}^{-1}\P\big[z_k\in\CT\Given B\big]\,\beta^k\\
&=
2\sum_{k=0}^\infty\beta^{-k}-\sum_{n=1}^\infty
\EE\left[\beta^{-(F_{-1}+F_{-1}'+1)+\ldots+(F_{-n}+F_{-n}'+1)}\right]- 1\\
&\qquad+\EE\left[\beta^{-(F_{-1}'+1)}+\beta^{-(F_{-1}'+F_{-1}+F_{-2}'+2)}+\ldots\right]\\
&=\frac{2\beta}{\beta-1}-\sum_{n=0}^\infty(s_-^2/\beta)^n+\frac{s_-}{\beta}\sum_{n=0}^\infty(s_-^2/\beta)^n\\
&=\frac{2\beta}{\beta-1}-\frac{1-s_-/\beta}{1-s_-^2/\beta}.
\end{aligned}
\end{equation}
Recall that
\begin{equation}\label{E2.08}
s_-=\EE\big[\beta^{-F_1}\big]=\sum_{k=0}^\infty(1-\alpha)\alpha^k\beta^{-k}=\frac{(1-\alpha)\beta}{\beta-\alpha}
\end{equation}
and
\begin{equation}\label{E2.09}
s_+=\EE\big[\beta^{+F_1}\big]=\sum_{k=0}^\infty(1-\alpha)\alpha^k\beta^{k}
=\cases{\DS\frac{1-\alpha}{1-\alpha\beta},&\mfalls \beta<\frac{1}{\alpha},\\[3mm]\infty,&\msonst.}
\end{equation}
Summing up, we have
\begin{equation}
\label{E2.10}
C=\frac{2\beta}{\beta-1}-\frac{1-s_-/\beta}{1-s_-^2/\beta}=\frac{2\beta}{\beta-1}-\frac{\beta-\alpha}{\beta-\alpha^2}.
\end{equation}
Let us now consider the details of the spanning around the position of the walker. That is, we consider the part of $\CT$ between two missing horizontal edges in which the walker is. Since we defined the position of the walker to be $0$, this is the part of $\CT$ between $H_0$ and $H_1$ (see Figure \ref{F1.2}). Recall that $F'_0$ is the distance of the rung to $H_0$ and $F_0$ is the distance of the rung to $H_1$. Furthermore, $V_0$ is the position of the rung. Note that $|W_1-W_0|=1$ if the missing horizontal edges are in alternating vertical positions and $W_1-W_0=0$ otherwise. We first condition on these random variables and compute the edge weights. Later we average over $F'_0$, $F_0$ and $|W_1-W_0|$.

Define the events
\begin{equation}
\label{E2.11}
\begin{aligned}
A_{a,b,k,\sigma}:&=\big\{F'_0=a,\,F_0=b,\,V_0=-k,\,|W_1-W_0|=\sigma\big\}\\
\end{aligned}
\end{equation}
for $a,b\in\N_0,\,k=-a,\ldots,b,\,\sigma=0,1$. For each of these events, we compute the conditional expectation $\E_0[\tau_1\given A_{a,b,k,\sigma}]$ and then sum over all possibilities.

Recall that $\phi(0)=(\phi^{(1)}(0),\phi^{(2)}(0))\in\Ray\cap\{(1,0),(0,0)\}$ is the position of the walker at time $0$. Note that given $A_{a,b,k,\sigma}$, we have $\phi(1)=(\phi^{(1)}(0),1))$ if $k\neq 0$ or $\sigma=0$ and $\phi(1)=(1-\phi^{(1)}(0),0)$ if $k=0$ and $\sigma=1$. In the latter case, the walker has to pass $z_0$ before it makes a jump to the right and accomplishes $\tau_1$.
In order to compute the expected waiting time before the walker makes a step on the ray, we use \equ{E2.06} and compute the edge weights $C_{\CT,-}$ ``to the left'' of the walker. More precisely, let $C_{\CT,-}$ denote the sum of the edge weights of all edges that can be connected to $\phi(0)$ without touching $\phi(1)$ (the next position right on the ray). In the case $k=0$, this includes the edges in the trap starting at $\phi(0)$.

\textbf{Case 1.} $k<0$.\\
Here we compute
\begin{equation}
\label{E2.12}
\begin{aligned}
\EE\big[C_{\tau,-}\given A_{a,b,k,\sigma}\big]
&=C\beta^{-a-k}+1+\beta^{-1}+\ldots+\beta^{1-a-k}\\
&=\left(C-\frac{\beta}{\beta-1}\right) \beta^{-a-k}+\frac{\beta}{\beta-1}.
\end{aligned}
\end{equation}
Denote by $\E_0$ the annealed expectation for the random walk started in $\phi(0)$. Since $C(h_{i,1})=\beta$, $i=1,2$, we conclude due to \equ{E2.06}
\begin{equation}
\label{E2.13}
\begin{aligned}
\E_{0}\big[\tau_1\Given A_{a,b,k,\sigma}\big]
&=\beta^{-1}\,2\,\E[C_{\tau,-}\given A_{a,b,k,\sigma}]+1\\
&=\frac{\beta+1}{\beta-1}+2\left(\frac{C}{\beta}-\frac{1}{\beta-1}\right)\beta^{-a-k}\\
&=:\frac{\beta+1}{\beta-1}+f_1(a,k).
\end{aligned}
\end{equation}

\textbf{Case 2.} $k>0$ or $k=0$ and $\sigma=0$.\\
This is quite similar to Case 1, but since $z_{-k}\in\CT$, we also have the additional edge weights
$$\beta^{-k}\big(1+\beta^{-a}+\beta^{-a+1}+\ldots+\beta^{b}\big).$$
Hence
\begin{equation}
\label{E2.14}
\begin{aligned}
\E_{0}\big[\tau_1\Given A_{a,b,k,\sigma}\big]
&=\frac{\beta+1}{\beta-1}+f_1(a,k)+2\beta^{-k}\left(\frac{1}{\beta}+\frac{1}{\beta-1}\big(\beta^b-\beta^{-a}\big)\right)\\
&=:\frac{\beta+1}{\beta-1}+f_1(a,k)+f_2(a,b,k).
\end{aligned}
\end{equation}

\textbf{Case 3.} $k=0$ and $\sigma=1$.\\
This works as in Case 2, but in addition, the walker has to pass the edge $z_0$.  Here we only compute the additional time that it takes for passing this edge. Since the edge weights
$$\beta+\ldots+\beta^{b}=\beta\frac{\beta^b-1}{\beta-1}$$
contribute to $C_{\CT,-}$, and since $C(z_0)=1$, we get
\begin{equation}
\label{E2.15}
\begin{aligned}
\E_{0}\big[\tau_1\given A_{a,b,k,\sigma}\big]
&=\frac{\beta+1}{\beta-1}+f_1(a,b)+f_2(a,b)+f_3(a,b)
\end{aligned}
\end{equation}
with
\begin{equation}
\label{E2.16}
\begin{aligned}
f_3(a,b)
&=2\left(C-\frac{\beta}{\beta-1}\right)\beta^{-a}+2\,\frac{\beta}{\beta-1}+1+2\beta\frac{\beta^b-1}{\beta-1}\\
&=1+2\left(C-\frac{\beta}{\beta-1}\right)\beta^{-a}+\frac{2\beta}{\beta-1}\beta^b.
\end{aligned}
\end{equation}

Now we sum up over the events $A_{a,b,k,\sigma}$. Recall the distribution of $G_0=F_0+F'_0+1$ from \equ{E1.02} and recall that $-H_0$ and $F'_0$ are independent are uniformly distributed on $\{0,\ldots,G_0-1\}$ given $G_0$. Hence
\begin{equation}
\label{E2.17}
\begin{aligned}
\PP\big[A_{a,b,k,\sigma}\big]&=\frac{1}{2}\PP\big[F'_0=a,\,F_0=b,\,-H_0=k+a\big]\\
&=\frac{1}{2}\frac{\PP[G_0=a+b+1]}{(a+b+1)^2}\\
&=\frac12\frac{1-\alpha}{1+\alpha}(1-\alpha)^2\alpha^{a+b}.
\end{aligned}
\end{equation}

Hence, we get
\begin{equation}
\label{E2.18}
\begin{aligned}
\sum_{a,b=0}^\infty&\sum_{k=-a}^b\sum_{\sigma=0}^1\PP\big[A_{a,b,k,\sigma}\big]\,f_1(a,k)\\
&=\frac{1-\alpha}{1+\alpha}\sum_{a,b=0}^\infty
\left(\frac{C}{\beta}-\frac{1}{\beta-1}\right)
(1-\alpha)^2\,\alpha^{a+b}\sum_{k=-a}^b\beta^{-a-k}\\
&=\frac{1-\alpha}{1+\alpha}\left(\frac{C}{\beta}
-\frac{1}{\beta-1}\right)\sum_{a,b=0}^\infty(1-\alpha)^2\alpha^{a+b}\ \frac{\beta-\beta^{-a-b}}{\beta-1}\\
&=\frac{1-\alpha}{1+\alpha}\,2\,\left(C-\frac{\beta}{\beta-1}\right)\frac{1-s_-^2/\beta}{\beta-1}.
\end{aligned}
\end{equation}

Furthermore,
\begin{equation}
\label{E2.19}
\begin{aligned}
\sum_{a,b=0}^\infty&\sum_{k=0}^b\sum_{\sigma=0}^1\PP[A_{a,b,k,\sigma}]f_2(a,k)\\
&=\frac{1-\alpha}{1+\alpha}\sum_{a,b=0}^\infty(1-\alpha)^2\,\alpha^{a+b}
\cdot2\cdot\left(\frac{1}{\beta}+\frac{1}{\beta-1}\big(\beta^b-\beta^{-a}\big)\right)\sum_{k=0}^b\beta^{-k}\\
&=\frac{1-\alpha}{1+\alpha}\sum_{a,b=0}^\infty(1-\alpha)^2\,\alpha^{a+b}
\cdot2\cdot\left(\frac{1}{\beta}+\frac{1}{\beta-1}\big(\beta^b-\beta^{-a}\big)\right)
\frac{\beta-\beta^{-b}}{\beta-1}\\
&=\frac{1-\alpha}{1+\alpha}\frac{2}{(\beta-1)^2}\Big[\beta s_++(\beta^{-1}-\beta-1)s_-+s_-^2+\beta-2\Big].
\end{aligned}
\end{equation}
Summing the terms for $f_1$ and $f_2$, we get the expected value for $\tau_1$ for the case where the ray stays either up all time or down all time. That is, by an explicit computation, we get
\begin{equation}
\label{E2.20}
\begin{aligned}
\E_0[\tau_1\Given W_n=1&\mbs{for all}n\in\Z]\\
&=\frac{\beta+1}{\beta-1}+\sum_{a,b,k,\sigma}\PP[A_{a,b,k,\sigma}]\,\big(f_1(a,k)+f_2(a,b,k)\big)\\
&=\frac{\beta+1}{\beta-1}+\frac{1-\alpha}{1+\alpha}\,2\,\left(C-\frac{\beta}{\beta-1}\right)
\frac{1-s_-^2/\beta}{\beta-1}\\
&\quad+\frac{1-\alpha}{1+\alpha}\frac{2}{(\beta-1)^2}\Big[\beta s_++(\beta^{-1}-\beta-1)s_-+s_-^2+\beta-2\Big]\\
&=\frac{1}{v^u}
\end{aligned}
\end{equation}
with $v^u$ from \equ{E2.03}.
The last equality is a tedious calculation that is omitted here. However, it is clear that $1/v^u$ must be the average holding time since the event we condition on is $\{\CT=\CT^u\}$.

Now we come to the last term that describes the slow down of the walker due to the need to pass through $z_{V_0}$ if $W_{-1}\neq W_0$.
\begin{equation}
\label{E2.21}
\begin{aligned}
\sum_{a,b}&\big(\E_0[\tau_1\Given A_{a,b,0,1}]-\E_0[\tau_1\Given A_{a,b,0,0}]\big)\,\PP[A_{a,b,0,1}]\\
&=\frac{1-\alpha}{1+\alpha}\cdot\frac12\sum_{a,b=0}^\infty(1-\alpha)^2\alpha^{a+b}f_3(a,b)\\
&=\frac12\frac{1-\alpha}{1+\alpha}\left[1+2\left(C-\frac{\beta}{\beta-1}\right)s_-+\frac{2\beta}{\beta-1}s_+\right].
\end{aligned}
\end{equation}
Concluding, we get the explicit formula (recall $C$ and $s_+$ from \equ{E2.07}, \equ{E1.11} and \equ{E1.10})
\begin{equation}
\label{E2.22}
\E_0[\tau_1]=\frac{1}{v^{u}}+\frac12\frac{1-\alpha}{1+\alpha}\left[1+2\left(C-\frac{\beta}{\beta-1}\right)s_-+\frac{2\beta}{\beta-1}s_+\right]
\end{equation}
and the formula for the speed
\begin{equation}
\label{E2.23}
v=\frac{1}{\E_0[\tau_1]}.
\end{equation}
Plugging in the expressions for $\E_0[\tau_1]$ and $v^{u}$ we get \equ{E1.12} and the proof of Theorem~\ref{T1} is complete. \hfill\qed

\Section{Ballistic central limit theorem, proof of Theorem~\ref{T2}}
\label{S3}
We follow the strategy of proof of Theorem 2 in \cite{Bowditch2019} by introducing regeneration times $(\sigma^X_n)_{n\in\N}$ (in \equ{E3.09} below) and showing that $\sigma^X_2-\sigma^X_1$ has a second moment.

Recall the definition of the random variables $H_n$, $n\in\Z$, from Section~\ref{S1.2}. Loosely speaking, $H_n$ is the right vertex of the $n$th horizontal bond right of the origin that has no matching horizontal bond. That is, either $\{(0,H_n-1),(0,H_n)\}$ is in the tree but $\{(1,H_n-1),(1,H_n)\}$ is not or vice versa. Let $i_n\in\{0,1\}$ be such that $e_n:=\{(i_n,H_n-1),(i_n,H_n)\}\in\CT$. Denote by $\CT_n$ the set of edges in $\CT$ between $H_n$ and $H_{n+1}$ shifted by $H_n$ and complemented by the information if the unique vertical edge between $H_{n-1}$ and $H_n$ is in this set is in the ray or not. More formally, for an edge $e=\{(j_1,h_1),(j_2,h_2)\}$, we define $e+(0,k)=\{(j_1,h_1+k),(j_2,h_2+k)\}$. Then
$$\CT_n:=\Big(\big\{e:\,e\subset\{0,1\}\times\{0,H_{n+1}-H_n-1\}:\,e+(0,H_n)\in\CT\big\},\1_{\{i_n\neq i_{n+1}\}}\Big).$$
Note that $(\CT_n)_{n=1,2,\ldots}$ is i.i.d.

Let $n_1<n_2<\ldots$ be the times where $X$ is on the ray:
$$\{n_1,n_2,\ldots\}=\{n:\,X_n\in \Ray\}.$$
Let $Y_k=X_{n_k}$, $k\in\N$. Then $Y$ is a random walk on the infinite ray of $\CT$.

Assume that $Y$ is started at $\phi(0)$. Since $Y$ is transient to the right, every point $\phi(n)$, $n>0$, is visited at least once. Recall that $C(\phi(n),\phi(n+1))$ is the conductance between $\phi(n)$ and $\phi(n+1)$. We write
\begin{equation}
\label{E3.01}
\mathcal{R}^\CT_{\mathrm{eff}}(n\leftrightarrow\infty)=\sum_{m=n}^\infty\frac{1}{C(\phi(m),\phi(m+1))}
\end{equation}
for the effective resistance between $\phi(n)$ and $+\infty$. By \cite[Theorem 19.25]{Klenke2020e}, the probability that $Y$ never returns to $\phi(n)$ is
\begin{equation}
\label{E3.02}
p^{\CT}(\phi(n))=\Big[\big(C\big(\phi(n-1),\phi(n)\big)+C\big(\phi(n),\phi(n+1)\big)\big)\,
\mathcal{R}^\CT_{\mathrm{eff}}(n\leftrightarrow\infty)\Big]^{-1}.
\end{equation}

Denote by $\ell^Y(\phi(n))=\sum_{k=0}^\infty \1_{\{\phi(n)\}}(Y_k)$ the local time of $Y$ at $\phi(n)$. The above considerations show that
\begin{lemma}
\label{L3.01}
The local time $\ell^Y(\phi(n))$ of $Y$ at $\phi(n)$, $n>0$, has distribution (given $\CT$)
\begin{equation}
\label{E3.03}
\delta_1\ast \gamma_{p^{\CT}(\phi(n))}
\end{equation}
with $p^{\CT}(\phi(n))$ given by \equ{E3.02}.
For $p^{\CT}(\phi(n))$ we have the bounds
\begin{equation}
\label{E3.04}
 \frac12\frac{\beta-1}{\beta+1}\,\leq\, p^{\CT}(\phi(n))\,\leq\, \frac{\beta-1}{\beta+1}.
\end{equation}
\end{lemma}
\begin{proof}
We only have to show \equ{E3.04}. Note that the effective resistance to $\infty$ is minimal if the ray is straight right of $\phi(n)$, that is, $\phi(n+k)=\phi(n)+(0,k)$ for $k=0,1,2,\ldots$. In this case, $C\big(\phi(n+k),\phi(n+k+1)\big)=\beta^k C\big(\phi(n),\phi(n+1)\big)$ and hence
$$\mathcal{R}^\CT_{\mathrm{eff}}\big(\phi(n)\leftrightarrow\infty\big)
\,=\,\frac{\beta}{\beta-1}\frac{1}{C\big(\phi(n),\phi(n+1)\big)}.$$
Since $C\big(\phi(n-1),\phi(n)\big)\geq C\big(\phi(n),\phi(n+1)\big)/\beta$, we get the upper bound in \equ{E3.04}, i.e.
$$p^{\CT}(\phi(n))\,\leq\, \frac{\beta-1}{\beta+1}.$$
The effective resistance is maximal if all the vertical edges are in the ray (to the right of $\phi(n)$, at least). Since each vertical edge has the same conductance as its preceding horizontal edge, the effective resistance essentially doubles. We distinguish the cases
\begin{itemize}
\item[(i)] where $\phi(n-1)$ and $\phi(n)$ are connected by a vertical edge and
\item[(ii)] where they are connected by a horizontal edge.
\end{itemize}
 In case (i), the edge between $\phi(n)$ and $\phi(n+1)$ is horizontal, the edge between $\phi(n+1)$ and $\phi(n+2)$ is vertical and so on. Summing up, we get
$$
\mathcal{R}^\CT_{\mathrm{eff}}\big(\phi(n)\leftrightarrow\infty\big)
\,=\,\frac{2\beta}{\beta-1}\frac{1}{C\big(\phi(n),\phi(n+1)\big)}$$
and $C\big(\phi(n-1),\phi(n)\big)=C\big(\phi(n),\phi(n+1)\big)/\beta$. That is
$$p^{\CT}(\phi(n))\,=\, \frac12\frac{\beta-1}{\beta+1}.$$
In case (ii), we have
$$\mathcal{R}^\CT_{\mathrm{eff}}\big(\phi(n)\leftrightarrow\infty\big)
=\frac{1}{C\big(\phi(n),\phi(n+1)\big)}+\frac{2}{\beta}\frac{\beta}{\beta-1}\frac{1}{C(n)}
=\frac{\beta+1}{\beta-1}\frac{1}{C(\phi(n),\phi(n+1))}$$
and $C\big(\phi(n-1),\phi(n)\big)=C\big(\phi(n),\phi(n+1)\big)$. That is, we get again
$$p^{\CT}(\phi(n))\,=\, \frac{1}{2C(n,n+1)\mathcal{R}^\CT_{\mathrm{eff}}\big(\phi(n)\leftrightarrow\infty\big)}
\,=\,\frac12\frac{\beta-1}{\beta+1}.$$
Summarizing, this gives the lower bound in \equ{E3.04}.
\end{proof}

As $Y$ has positive speed to the right, it passes each $e_n$, $n\in\N$ at least once. We say that $n$ is a regeneration point if $Y$ passes $e_n$ exactly once. Note that for given $\CT$ and $n$, the probability that $n$ is regeneration point is
\begin{equation}
\label{E3.05}
\begin{aligned}
\P^\CT_{(i_n,H_n)}\big[Y_k\neq (i_n,H_n-1)\mbox{ for all }k\geq1\big]&=\frac{\mathcal{R}^\CT_{\mathrm{eff}}\big((i_n, H_n-1)\leftrightarrow (i_n,H_n)\big)}{\mathcal{R}^{\CT}_{\mathrm{eff}}\big((i_n, H_n-1)\leftrightarrow \infty\big)}\\
&=\frac{\beta^{-H_n}}{\mathcal{R}^{\CT}_{\mathrm{eff}}\big((i_n, H_n-1)\leftrightarrow \infty\big)}\\
&\geq \left(2\sum_{k=0}^\infty\beta^{-k}\right)^{-1}=\frac{\beta-1}{2\beta}.
\end{aligned}
\end{equation}
Let $$\sigma^Y_n:=\inf\{k:\,Y_k=(i_n,H_n)\}.$$
For $m>n$, let
\begin{equation}
\label{E3.06}
\begin{aligned}
A_n&:=\big\{n\mbox{ is a regeneration point}\big\}\\
&=\Big\{Y_k\not\in\{0,1\}\times\{H_n-1,H_n-2,\ldots\}\mbox{ for all }k\geq \sigma_n^Y\Big\}
\end{aligned}
\end{equation}
and
\begin{equation}
\label{E3.07}
A_{n,m}:=\Big\{Y_k\not\in\{0,1\}\times\{H_n-1,H_n-2,\ldots\}\mbox{ for all }k\in\{\sigma^Y_n,\ldots,\sigma^Y_m\}\Big\}.
\end{equation}
Clearly, we have $A_{n}\subset A_{n,m}$ for all $m>n$. Now let $n_1,\ldots,n_k\in\N$, $n_1<n_2<\ldots<n_k$. Note that $A_{n_l,n_{l+1}}$, $l=1,\ldots,k-1$, and $A_{n_k}$ are independent events under $\P^\CT_0$. Hence
\begin{equation}
\label{E3.08}
\begin{aligned}
\P^\CT_0\big[n_l&\mbox{ is a regeneration point for all }l=1,\ldots,k\big]\\
&=\P^\CT_0\left[\bigcap_{l=1}^k A_{n_l}\right]\\
&=\P^\CT_0\left[\left(\bigcap_{l=1}^{k-1} A_{n_l,n_{l+1}}\right)\cap A_{n_k}\right]\\
&=\left(\prod_{l=1}^{k-1}\P^\CT_{0}\left[A_{n_l,n_{l+1}}\right]\right)\cdot\P^\CT_{0} [A_{n_k}]\\
&\geq\prod_{l=1}^{k}\P^\CT_{0}\left[A_{n_l}\right]\\
&\geq\left(\frac{\beta-1}{2\beta}\right)^k.
\end{aligned}
\end{equation}
Summing up, the set of regeneration points is minorized by a Bernoulli point process with success probability $(\beta-1)/2\beta$ that is independent of $\CT$. More explicitly, there exist i.i.d.{} Bernoulli random variables $Z_1,Z_2,\ldots$ independent of $\CT$ such that $\P[Z_1=1]=(\beta-1)/2\beta$ and such that $Z_n=1$ implies that $n$ is a regeneration point (but not vice versa).

Now let $\tau_1<\tau_2<\ldots$ denote the sequence of $n$'s such that $Z_n=1$. Note that $(\tau_{k+1}-\tau_{k})_{k\in\N}$ are i.i.d.{} geometric random variables. Let
\begin{equation}
\label{E3.09}
\sigma^X_n:=\inf\big\{k:\,X_k=(i_{\tau_n},H_{\tau_n})\big\}.
\end{equation}
and
\begin{equation}
\label{E3.10}
L_k:=\#\big\{m\in\N_0:\, X_m\in\{0,1\}\times\{H_{k},\ldots,H_{k+1}-1\}\big\}.
\end{equation}
Note that for $n\in\N$,
\begin{equation}
\label{E3.11}
\sigma^X_{n+1}-\sigma^X_n= \sum_{k=\tau_n}^{\tau_{n+1}-1}L_k
\end{equation}

Then $(\sigma^X_{k+1}-\sigma^X_k)_{k\in\N}$ are i.i.d.{} random variables (under the annealed probability measure). Along each elementary block $\CT_n$ the ray passes $(\#\CT_n(1)+1)/2=H_{n+1}-H_n$ horizontal edges and $\CT_n(2)$ vertical edges. Note that
\begin{equation}
\label{E3.12}
\begin{aligned}
\PP\big[\#\CT_n(1)+1= 2k\big]&=\PP[H_{n+1}-H_n=k]\\
&=\gamma_{1-\alpha}\ast\gamma_{1-\alpha}(k-1)\\
&=(1-\alpha)^2\,k\alpha^k.
\end{aligned}
\end{equation}
For each step of $Y_n$, the random walk $X_n$ spends a random amount of time $T_n$ in a trap. For most points of the ray, this random variable is simply $1$. Only at points adjacent to a vertical edge, there can be either one or two real traps with $T_n$ possibly larger than $1$. Let
$$\varrho=-\frac{\log(\alpha)}{\log(\beta)}$$
and recall that $\varrho>2$ since $1<\beta<\beta_c^{(2)}=1/\sqrt{\alpha}$.
By Theorem 1.1 of \cite{GantertKlenke2022}, there exists a constant $K$ such that
$$\P[T_n>t]\leq Kt^{-\varrho}\mfa t>0.$$
Hence, for all $\zeta<\varrho$, we have $\E[T_n^\zeta]<\infty$. In fact, the result of \cite{GantertKlenke2022} gives a precise statement for the tails of the traps of type (a) (see Figure~\ref{F2.1}) only, but clearly, the cases (b) and (c) work similarly.

Let
\begin{equation}
\label{E3.13}
\ell^Y\big((i,m)\big):=\#\big\{n\in\N_0:Y_n=(i,m)\big\}
\end{equation}
be the occupation time of $Y$ in $(i,m)$. Given $\CT$, for each $(i,m)$ on the ray, $\ell^Y\big((i,m)\big)-1$ is a geometric random variable with success probability
$p^{\CT}\big((i,m)\big)$. Here $p^{\CT}\big((i,m)\big)$ is the probability that $Y$ will go to infinity before returning to $(i,m)$. By Lemma~\ref{L3.01}, this probability is bounded below by
$$p^{\CT}\big((i,m)\big)\geq \frac12\frac{\beta-1}{\beta+1}.$$
Let $\gamma_p(k)=p(1-p)^k$ denote the weights of the geometric distribution on $\N_0$ with success probability $p$. Hence
for every $\zeta>0$, there is a constant $K_\zeta<\infty$ such that
\begin{equation}
\label{E3.14}
\E^\CT\Big[\big(\ell(i,m)\big)^\zeta\Big]\leq K_\zeta:=\sum_{k=1}^\infty \gamma_{(\beta-1)/{2(\beta+1)}}(k-1)k^\zeta<\infty.
\end{equation}
Note that $H_{n+1}-H_n$ has moments of all orders, hence there exists $K'_\zeta<\infty$ such that (recall $L_n$ from \equ{E3.10})
\begin{equation}
\label{E3.15}
\begin{aligned}
\EE[L_n^\zeta]\leq K'_\zeta<\infty.
\end{aligned}
\end{equation}
Now fix $\zeta\in(2,\varrho)$ and let $\eta=\frac{\zeta}{\zeta-2}$. Note that $\frac1{\zeta/2}+\frac1\eta=1$. Recall that $\tau_{n+1}-\tau_n$ is geometrically distributed. We use Hölder's inequality to infer
\begin{equation}
\label{E3.16}
\begin{aligned}
\E\Big[\big(\sigma^X_{2}&-\sigma^X_2\big)^2\Big]
=\E\left[\left(\sum_{k=\tau_1}^{\tau_{2}-1}L_k\right)^2\right]\\
&=\sum_{i=1}^\infty\sum_{\ell=1}^\infty\E\left[\left(\sum_{k=i}^{i+\ell-1}L_k\right)^2;\,\tau_1=i,\,\tau_{2}=i+\ell\right]\\
&\leq\sum_{i=1}^\infty\sum_{\ell=1}^\infty\E\left[\left(\sum_{k=i}^{i+\ell-1}L_k\right)^\zeta\right]^{2/\zeta}\P[\tau_1=i,\,\tau_{2}=i+\ell]^{1/\eta}\\
&\leq\max_{k\geq 1}\E\left[L_k^\zeta\right]^{2/\zeta}\sum_{i=1}^\infty\sum_{\ell=1}^\infty\ell^2\P[\tau_1=i,\,\tau_{2}=i+\ell]^{1/\eta}\\
&\leq (K'_\zeta)^{2/\zeta}\left(\frac{\beta-1}{2\beta}\right)^{2/\eta}\left(\sum_{\ell=1}^\infty\left(\frac{\beta+1}{2\beta}\right)^{(\ell-1)/\eta}\ell^2\right)\left(\sum_{\ell=1}^\infty\left(\frac{\beta+1}{2\beta}\right)^{(\ell-1)/\eta}\right)<\infty.
\end{aligned}
\end{equation}
Having established the existence of second moments of the regeneration times, we can argue as in the proof of Theorem 2 in \cite{Bowditch2019} to conclude the proof of Theorem~\ref{T2}.
\hfill\qed
\Section{Einstein relation: proof of Theorem~\ref{T3}}
\label{S4}
Recall that $\Ray$ denotes the (unique) self-avoiding path on $\CT$ from $-\infty$ to $\infty$. We may and will assume that $X$ starts at a point chosen from the ray instead of (possibly) from a point inside a trap. Since all traps are of finite size, it is enough to prove the theorem for this initial position. For $n\in\N_0$, define the last time $m\leq n$ when $X_m$ was on the ray by
$$n^*:=\inf\{m\leq n:\,X_m\in\Ray\}$$
and let
$$\tilde X_n:=X_{n^*}.$$
That is, $\tilde X$ waits at the entrances of the traps for $X$ to come back to the ray. Since the depths of the traps are independent geometric random variables, it is easy to check that
$$\frac{1}{\sqrt{n}}|X_n^{(1)}-\tilde X_n^{(1)} |\limn0\mfs$$
Hence, it is enough to show the theorem for $\tilde X$ instead of $X$.

Note that $\tilde X$ is symmetric simple random walk on $\Ray$ with random holding times. We now give a different construction of such a random walk that allows explicit computations.

Recall that $\phi(n)=(\phi^{(1)}(n),\phi^{(2)}(n))$, $n\in\Z$, is the enumeration of $\Ray$, such that $\phi^{(1)}(0)=0$ and $\phi^{(1)}(-1)=-1$. Let $Y$ be symmetric simple random walk on $\Z$. Then $\phi(Y_n)$ is symmetric simple random walk on $\Ray$. For each $k\in\Z$, there may or may not be a trap starting at $\phi(k)$. Let $\nu_k$ denote the (random) distribution of the time that $X$ spends in that trap at any visit of $\phi(k)$ before it makes a move on $\Ray$. Let $U_{k,i}$ denote the time that $\tilde X$ spends at its $i$th visit to $\phi(k)$ before it makes a step on the ray. Clearly, the $U_{k,i}$ are independent given $(\nu_k)_{k\in\Z}$ and $U_{k,i}$ has distribution $\nu_k$. Now we define
$$S_n:=\sum_{k=0}^n U_{Y_k,k}$$ and let
$$Y^D_n=Y_k\mfalls S_{k-1}<n\leq S_{k}.$$
Then $\phi(Y^D_n)_{n\in\N_0}$ and $(\tilde X_n)_{n\in\N_0}$ coincide in (quenched) distribution.
Hence the proof of Theorem~\ref{T3} amounts to showing that
\begin{equation}
\label{E4.01}
\left(\frac{1}{\sqrt{\sigma^2\,n}}\phi^{(1)}(Y^D_n)_{\lfloor tn\rfloor}\right)_{t\geq0}
\end{equation}
converges in $\P^\CT$-distribution in the Skorohod space $D_{[0,\infty)}$ to a standard Brownian motion.

Let $m:=\EE[\sum_\ell\nu_k(\{\ell\})\ell]$ be the mean time before $Y^D$ takes its next step. In Theorem 2.10 of \cite{BenArousCerny2015}, a functional central limit theorem for our $Y^D$ was shown for the situation where the sequence $(\nu_n)_{n\in\Z}$ is i.i.d. In our situation, the $(\nu_n)_{n\in\Z}$ are not i.i.d.{} but are ergodic. In fact, they are a strongly mixing sequence as the random spanning tree $\CT$ is made of i.i.d.{} building blocks. Taking a closer look at the proof of Theorem 2.10 in \cite[Section 7.2]{BenArousCerny2015}, we notice that their assumption of independence is too strong and that ergodicity is perfectly enough to infer the statement of the theorem. Hence, by Theorem 2.10 of \cite{BenArousCerny2015}, we get that
\begin{equation}
\label{E4.02}
\left(\frac{1}{\sqrt{n/m}}Y^D_{\lfloor tn\rfloor}\right)_{t\geq0}\end{equation}
converges in $\P^\CT$-distribution in the Skorohod space $D_{[0,\infty)}$
to a standard Brownian motion.

It remains to compute $m$ and to measure the effect of applying $\phi^{(1)}$ to $Y^D$.

\begin{lemma}
\label{L4.01}
For almost all $\CT$, we have
$$\lim_{|n|\to\infty}\frac{\phi^{(1)}(n)}{n}=\frac{2(1+\alpha)}{3+\alpha}.$$
\end{lemma}
\textbf{Proof. }
Let $$p:=\PP[\CT\ni z_{k}]=\PP[V_n=k\mbs{for some}n\in\Z]=\frac{1-\alpha}{1+\alpha}, \qquad k\in\Z,$$
be the probability that a given vertical edge $z_k$ is in the spanning tree. Recall that the fair coin flips $(W_n)$ determine the vertical positions of the missing edges. That is, the horizontal edge $h_{W_n,H_n}$ is not in the tree $\CT$. Note that $H_n\leq V_n<H_{n+1}$ and that the vertical edge $z_{V_n}$ is in the ray if and only if $W_n\neq W_{n+1}$. Hence
$$\PP[\Ray\ni z_{k}]=\sum_{n\in\Z}\P[V_n=k,\,W_n\neq W_{n+1}]=\frac12\sum_{n\in\Z}\P[V_n=k]=\frac{p}{2}.$$

By ergodicity, we get for almost all $\CT$
$$\frac1n\#\big\{k\in\{0,\ldots,n\}:\,z_k\in\Ray\big\}\limn \frac p2$$
and
$$\frac1n\#\big\{k\in\{-n,\ldots,0\}:\,z_k\in\Ray\big\}\limn \frac p2.$$
This implies
\[\lim_{|n|\to\infty}\frac{\phi^{(1)}(n)}{n}=\frac{1}{1+p/2}=\frac{2(1+\alpha)}{3+\alpha}.
\tag*{$\Box$}\]
\medskip
\par

\begin{lemma}
\label{LemmaCLT2}
We have
$$m=\frac{4(\alpha+1)}{3+\alpha}.$$
\end{lemma}
\textbf{Proof. }
For $n$ such that the vertex $\phi(n)$ is of degree 2 (that is, there is no trap adjacent to $\phi(n)$), we have $\nu_n=\delta_1$ since $Y^D$ makes its next step immediately. For $n$ such that $\phi(n)$ has degree 3 and there is a trap consisting of $k$ edges starting at $\phi(n)$, by \equ{E2.05}, we have that the average holding time is
$$\sum_\ell\nu_n(\ell)\ell=k+1.$$
By the ergodic theorem, we get
\begin{align}
m&=\lim_{n\to\infty}\frac1n\sum_{k=0}^n\sum_{\ell}\nu_k(\ell)\ell\nonumber\\\nonumber
&=\lim_{n\to\infty}\frac1n\#\big\{\mbox{edges of }\CT\cap(\{0,1\}\times\{0,\ldots,\phi^{(1)}(n)\})\big\}\\
&=\lim_{n\to\infty}\frac1n2\phi^{(1)}(n)=\frac{4(1+\alpha)}{3+\alpha}.
\tag*{$\Box$}
\end{align}
\medskip\par

Together with the above discussion, these lemmas finish the proof of Theorem~\ref{T3}.
\hfill\qed

\bibliographystyle{plain}

\begin{thebibliography}{10}

\bibitem{Aidekon2014}
Elie A\"{\i}d\'{e}kon.
\newblock Speed of the biased random walk on a {G}alton-{W}atson tree.
\newblock {\em Probab. Theory Related Fields}, 159(3-4):597--617, 2014.

\bibitem{BalakBroeck}
Venkataraman Balakrishnan, Christian Van den Broeck
\newblock {Transport properties on a random comb}
\newblock {\em Physica A}, 217: 1--21, 1995.

\bibitem{BarmaDhar}
Mustansir Barma and Deepak Dhar.
\newblock Directed diffusion in a percolation network.
\newblock {\em Journal of Physics C}, 16(8), 1983.

\bibitem{BenArousCerny2015}
G\'{e}rard Ben~Arous, Manuel Cabezas, Ji\v{r}\'{\i} \v{C}ern\'{y}, and Roman
  Royfman.
\newblock Randomly trapped random walks.
\newblock {\em Ann. Probab.}, 43(5):2405--2457, 2015.



\bibitem{GBAFri}
G\'{e}rard Ben~Arous and Alexander Fribergh.
\newblock Biased random walks on random graphs.
\newblock In {\em Probability and statistical physics in {S}t. {P}etersburg},
  volume~91 of {\em Proc. Sympos. Pure Math.}, pages 99--153. Amer. Math. Soc.,
  Providence, RI, 2016.

\bibitem{BFGH}
G\'{e}rard Ben~Arous, Alexander Fribergh, Nina Gantert, and Alan Hammond.
\newblock Biased random walks on {G}alton-{W}atson trees with leaves.
\newblock {\em Ann. Probab.}, 40(1):280--338, 2012.

\bibitem{BenArousFriberghSidorvacius2014}
G\'{e}rard Ben~Arous, Alexander Fribergh, and Vladas Sidoravicius.
\newblock Lyons-{P}emantle-{P}eres monotonicity problem for high biases.
\newblock {\em Comm. Pure Appl. Math.}, 67(4):519--530, 2014.

\bibitem{BergerGantertNagel2019}
Noam Berger, Nina Gantert, and Jan Nagel.
\newblock The speed of biased random walk among random conductances.
\newblock {\em Ann. Inst. Henri Poincar\'{e} Probab. Stat.}, 55(2):862--881,
  2019.

\bibitem{BGP}
Noam Berger, Nina Gantert, and Yuval Peres.
\newblock The speed of biased random walk on percolation clusters.
\newblock {\em Probab. Theory Related Fields}, 126(2):221--242, 2003.

\bibitem{BetzMeinersTomic2023}
Volker Betz, Matthias Meiners, and Ivana Tomic.
\newblock Speed function for biased random walks with traps.
\newblock {\em Statistics \& Probability Letters}, 195:109765, 2023.


\bibitem{Bowditch2019}
Adam Bowditch.
\newblock Central limit theorems for biased randomly trapped random walks on
  {$\Bbb Z$}.
\newblock {\em Stochastic Process. Appl.}, 129(3):740--770, 2019.

\bibitem{BowditchCroydon2022}
Adam~M. Bowditch and David~A. Croydon.
\newblock Biased random walk on supercritical percolation: anomalous
  fluctuations in the ballistic regime.
\newblock {\em Electron. J. Probab.}, 27:Paper No. 68, 22, 2022.

\bibitem{DemaerelMaes}
Thibaut Demaerel ad Christian Maes.
\newblock The asymptotic speed of reaction fronts in active reaction-diffusion systems.
\newblock {\em Journal of Physics A}, 52, 2019, 245001.

\bibitem{FriHam}
Alexander Fribergh and Alan Hammond.
\newblock Phase transition for the speed of the biased random walk on the
  supercritical percolation cluster.
\newblock {\em Comm. Pure Appl. Math.}, 67(2):173--245, 2014.

\bibitem{GantertKlenke2022}
Nina Gantert and Achim Klenke.
\newblock The {T}ail of the {L}ength of an {E}xcursion in a {T}rap of {R}andom
  {S}ize.
\newblock {\em J. Stat. Phys.}, 188(3):Paper No. 27, 2022.

\bibitem{GMM2}
Nina Gantert, Matthias Meiners, and Sebastian M\"{u}ller.
\newblock Einstein relation for random walk in a one-dimensional percolation
  model.
\newblock {\em J. Stat. Phys.}, 176(4):737--772, 2019.

\bibitem{Haggstrom1994}
Olle H{\"a}ggstr{\"o}m.
\newblock {\em Aspects of Spatial random processes}.
\newblock PhD thesis, University G{\"o}te\-borg, 1994.

\bibitem{Hammond}
Alan Hammond.
\newblock Stable limit laws for randomly biased walks on supercritical trees.
\newblock {\em Ann. Probab.}, 41(3A):1694--1766, 2013.

\bibitem{Klenke2017}
Achim Klenke.
\newblock {The random spanning tree on ladder-like graphs}.
\newblock {\em https://doi.org/10.48550/arXiv.1704.00182}, 2017.

\bibitem{Klenke2020e}
Achim Klenke.
\newblock {\em Probability theory: {A} comprehensive course.}
\newblock Universitext. Springer Nature Switzerland AG, Cham, 3. edition, 2020.

\bibitem{KotakBarma}
Jesal Kotak and Mustansir Barma.
\newblock {Biased induced drift and trapping on randpom combs and the Bethe lattice: Fluctuation regime and first order phase transitions}
\newblock {\em Physica A}, 597, 2022, 127311.



\bibitem{MeiLue}
Jan-Erik L\"{u}bbers and Matthias Meiners.
\newblock The speed of critically biased random walk in a one-dimensional
  percolation model.
\newblock {\em Electron. J. Probab.}, 24:Paper No. 23, 29, 2019.

\bibitem{LPP}
Russell Lyons, Robin Pemantle, and Yuval Peres.
\newblock Biased random walks on {G}alton-{W}atson trees.
\newblock {\em Probab. Theory Related Fields}, 106(2):249--264, 1996.

\bibitem{LyonsPeresPemantle1997}
Russell Lyons, Robin Pemantle, and Yuval Peres.
\newblock Unsolved problems concerning random walks on trees.
\newblock In {\em Classical and modern branching processes ({M}inneapolis,
  {MN}, 1994)}, volume~84 of {\em IMA Vol. Math. Appl.}, pages 223--237.
  Springer, New York, 1997.

\bibitem{Pottier}
No\"{e}lle Pottier
\newblock {Diffusion on random comblike structures: field-induced trapping effects}
\newblock {\em Physica A}, 216: 1--19, 1995.

\bibitem{Sznitperc}
Alain-Sol Sznitman.
\newblock On the anisotropic walk on the supercritical percolation cluster.
\newblock {\em Comm. Math. Phys.}, 240(1-2):123--148, 2003.

\bibitem{WhiteBarma}
Steven R.  White and Mustansir Barma.
\newblock Field-induced drift and trapping in percolation networks.
\newblock {\em Journal of Physics A}, 17, 2995 -- 3008, 1984.


\end{thebibliography}

Data Availability Statement:\\
Data sharing not applicable to this article as no datasets were generated or analysed during the current study.\\
Conflict of interest:\\
There is no conflict of interest.\\
Funding:\\
The authors did not receive support from any organization for the submitted work.

\end{document}